\documentclass[12pt, reqno, a4paper,oneside]{amsart}
\pdfoutput=1
\usepackage{graphicx, epsfig, psfrag}
\usepackage{amsfonts,amsmath,amssymb,amsbsy,amsthm}
\usepackage{bm}
\usepackage{color}
\usepackage{float}
\usepackage{mathrsfs}
\usepackage[left=2.5cm,right=2.5cm,top=2.5cm,bottom=2.5cm]{geometry}
\usepackage{longtable}
\usepackage{environments}

\renewcommand{\cases}[1]{\left\{ \begin{array}{rl} #1 \end{array} \right.}
\newcommand{\smfrac}[2]{{\textstyle \frac{#1}{#2}}}

\def\Xint#1{\mathchoice
{\XXint\displaystyle\textstyle{#1}}%
{\XXint\textstyle\scriptstyle{#1}}%
{\XXint\scriptstyle\scriptscriptstyle{#1}}%
{\XXint\scriptscriptstyle\scriptscriptstyle{#1}}%
\!\int}
\def\XXint#1#2#3{{\setbox0=\hbox{$#1{#2#3}{\int}$ }
\vcenter{\hbox{$#2#3$ }}\kern-.6\wd0}}
\def\mint{\Xint-}

\usepackage{accents}
\newlength{\dhatheight}
\newlength{\dtildeheight}

\def\b{\big}
\def\B{\Big}
\def\bg{\bigg}

\def\sep{\,|\,}

\def\R{\mathbb{R}}
\def\N{\mathbb{N}}

\def\dt{\,{\rm d}t}

\def\dd{{\rm d}}

\def\<{\langle}
\def\>{\rangle}

\def\D{\nabla}

\def\E{\mathcal{E}}

\definecolor{cocol}{rgb}{0.7, 0, 0}
\definecolor{nickcol}{rgb}{0, 0.7, 0}
\definecolor{davecol}{rgb}{0, 0, 0.7}

\begin{document}

\title[An Efficient Dimer Method]{An Efficient Dimer Method with
  Preconditioning and Linesearch}

\author{N. Gould}
\address{N.I.M.\ Gould \\ Scientific Computing Department \\
STFC-Rutherford Appleton Laboratory \\ Chilton, OX11 0QX \\ UK}
\email{nick.gould@stfc.ac.uk}

\author{C. Ortner}
\address{C. Ortner\\ Mathematics Institute \\ Zeeman Building \\
  University of Warwick \\ Coventry CV4 7AL \\ UK}
\email{christoph.ortner@warwick.ac.uk}

\author{D. Packwood}
\address{D. Packwood\\ Mathematics Institute \\ Zeeman Building \\
  University of Warwick \\ Coventry CV4 7AL \\ UK}
\email{d.packwood@warwick.ac.uk}

\date{\today}

\thanks{This work was supported in part by EPSRC grants EP/J021377/1
  and EP/J022055/1.}

\subjclass[2000]{65K99, 90C06, 65Z05}


\keywords{saddle search, perconditioning, convergence, dimer method}

\begin{abstract}
  The dimer method is a Hessian-free algorithm for computing saddle
  points. We augment the method with a linesearch mechanism for
  automatic step size selection as well as preconditioning
  capabilities. We prove local linear convergence. A series of
  numerical tests demonstrate significant performance gains.
\end{abstract}

\maketitle


\section{Introduction}

The problem of determining saddle points on high dimensional surfaces
has received a great deal of attention from the chemical physics
community over the past few decades. These surfaces arise, in particular,
as potential energies of molecules or materials. The local minima of
such functions describe stable atomistic configurations, while saddle
points provide information about the transition rates between minima in the
harmonic approximation of transition state theory. Independently, they
are useful for mapping the energy landscape and are used to inform
accelerated MD type schemes such as hyperdynamics
\cite{Voter:prl1997,Uberuaga:2005} or kinetic Monte Carlo (KMC)
\cite{Voter:2007}.

While the problem of determining the minima of such an energy function
is well known in the numerical analysis community, the problem of
locating saddles point has received little attention. 
Saddle search algorithms can be broadly categorised into two groups.

The first group has been called `chain of states' methods. A chain of
`images' are placed on the energy surface, often the two end points of
the chain are placed at two different local minima, for which the
connecting saddle is being sought. The chain is then `relaxed' by some
dynamics for which the mininum energy path (MEP) is (thought to be) an
attractor. Two archetypical methods of this class are the nudged elastic
band (NEB) method \cite{Jonsson:1998} and the string method
\cite{Weinan:prb2002,Weinan:jcp2007}.

The second group of methods for finding the saddle have been called
`walker' methods. Here a single `image' moves from its initial point
(sometimes, but not obligatorily, a local minimum) until it becomes
sufficiently close to a saddle point. The first method to work in this
framework was Rational Function Optimization (RFO) and later its
derivative, the Partitioned RFO (PRFO)\cite{Cerjan:jcp1981,
  Simons:jpc1983, Banerjee:jpc1985}. Here, the full eigenstructure of
the Hessian is explicitly calculated and then one or more eigenvalues
are manually shifted. In particular, if the minimum eigenvalue is
shifted in the correct manner, and a Newton step is applied using the
resultant modified Hessian, then the walker moves uphill in the
direction corresponding to the lowest eigenvector and downhill in all
other directions. If the Hessian is expensive to calculate, or even
unavailable, it can be approximated as the computation proceeds 
by any variety of techniques, for example the symmetric rank-one 
approximation \cite{Murtagh:cj1970}. Of course any useful Hessian
approximation should necessarily have the flexibility to be
indefinite. Other walker type techniques are satisfied with computing
the lowest eigenpair only. One such technique is the Activation Relaxation
Technique (ART) nouveau \cite{Marinica:prb2011,Machado-Charry:jcp2011,
  Mousseau:jamop2012,Cances:jcp2009}.  The original ART method used an
ascent step not along the minimum eigenvector, but along a line drawn
between the image and a known local minimum
\cite{Barkema:prl1996,Barkema:cms2001}. In ART nouveau this is
replaced by the minimum eigenpair which is calculated by means of the
Lanczos \cite{Lanc50} method.

The technique which forms the basis of the present paper, is the {\em
  dimer method} \cite{HenkJons:jcp1999,Heyden:jcp2005}. In this method
a pair of `walkers' is placed on the energy surface and aligned with
the minimum eigenvector (irrespective of the sign of the corresponding
eigenvalue) by minimizing the sum of the energies at the two end
points. This can be thought of as the computation of the minimal
eigenvalue using a finite difference approximation to the Hessian
matrix. In practice this `rotation step' is not converged
to great precision. More advanced modifications 
can be used to improve walker search directions, e.g., an L-BFGS
\cite{LiuNoce89} scaling, rather than just using a default steepest
descent type scheme \cite{Kastner:jcp2008}.

In the only rigorous analysis of the dimer method that we are aware of
Zhang and Du \cite{ZhangDu:sinum2012} prove local convergence of a
variation where the `dimer length' (the separation distance between
the two walkers) shrinks to zero. In that work the dimer evolution is
treated as a dynamical system, and the stability of different types of
equilibria is investigated.

In the present paper we present three new results:
\begin{enumerate}
\item We augment the dimer method with preconditioning capabilities to
  improve its efficiency for ill-conditioned problems, in particular
  with an eye to high-dimensional molecular energy landscapes. This
  modification is based on the elementary observation that the dimer
  method can be formulated with respect to an arbitrary inner product.
  (The $\ell^2$-inner product was previously used exclusively.)

\item We introduce a linesearch procedure. To that end, the main
  difficulty is the absence of a merit function for saddles. Instead,
  we proposed a {\em local merit function}, which we minimise at each
  dimer iteration using traditional linesearch strategies from
  optimisation, and which is updated between steps.

\item We present a variation of the analysis of Zhang and Du
  \cite{ZhangDu:sinum2012} that demonstrates that it is unnecessary to
  shrink the dimer length, $h$, to zero. Indeed, shrinking $h$ can
  cause severe numerical difficulties due to round-off. We prove that,
  if it is kept fixed, then the dimer walkers converge to a point that
  lies within $O(h^2)$ of a saddle. We also extend this analysis to
  incorporate preconditioning and linesearch.
\end{enumerate}


Concerning (2), it would of course be preferable to construct a global
merit function as this would provide a path towards constructing a
globally convergent scheme. Indeed, our (non-trivial) generalisation
of the convergence analysis to the linesearch variant of the dimer
method only yields local results, and we even present counterexamples
to global convergence.

The paper is organised as follows: having established preliminary
concepts, we describe two variants of the basic dimer method, and
establish their local convergence, in \S2. A linesearch enhancement is
proposed, and its local convergence behaviour is analysed, in
\S3. Numerical experiments illustrating the advantages of the
linesearch are given in \S4. We conclude in \S5.  Full details of our
analysis are given in Appendix A.

\section{Local Convergence of the Dimer Method}

\subsection{Preliminaries}
Let $X$ be a Hilbert space with norm $\| x \|$ and inner product $x
\cdot y$. We write $x \perp y$ if $x \cdot y = 0$. $I : X \to X$
denotes the identity. For $x, y \in X$, $x\otimes y : X \to X$ denotes
the operator defined by $(x \otimes y) z = (y\cdot z) x$.

Given two real functions $f$ and $g$ defined in some neighbourhood 
$\mathcal{N}$ of the origin, we say that $f(x) = O(g(x))$ 
as $x \rightarrow 0$ if $|f(x)| \leq C |g(x)|$
for some constant $C > 0$ and all $x \in \mathcal{N}$.

For a bounded linear operator $A \in L(X)$ we denote its spectrum by
$\sigma(A)$. We say that $(\lambda, v) \in \R \times X$ is an
eigenpair if $A v = \lambda v$. If $(\lambda, v)$ is an eigenpair and
$\lambda = \inf \sigma(A)$, then we call it a minimal eigenpair. We
say that $A$ has {\em index-1 saddle structure} if there exists a
unique minimal eigenpair $(\lambda, v)$ with $\lambda < 0$ and $A$ is
positive definite in $\{ v \}^\perp$.

If $F : X \to \R$ is Fr\'echet differentiable at a point $x$ then we
denote its {\em gradient} by $\D F(x)$, i.e., 
\begin{displaymath}
  \D F(x) \cdot y = \lim_{t \to 0} t^{-1}(F(x+t y) - F(x)).
\end{displaymath}
(Note that $\D F(x)$ is the Riesz representation of the first
variation $\delta F(x) \in X^*$.) Similarly, if $F : X \to X$ is
Fr\'echet differentiable at $x$, then $\D F(x) \in L(X)$ is a bounded
linear operator satisfying $\D F(x) u = \lim_{t\to 0} t^{-1} (F(x+t u)
- F(x))$. In particular, if $F : X \to \R$, then the Hessian $\D^2
F(x) \in L(X)$ (rather than $\D^2 F(x) : X \to X^*$). Higher
derivatives are defined analogously, but we shall avoid their explicit
use as much as possible.

%
We say that $x_*$ is an {\em index-1 saddle} of $E$ if
\begin{equation}
  \label{eq:defn:index1-saddle}
  \D E(x_*) = 0 \quad \text{and} \quad
  \D^2 E(x_*) \text{ has index-1 saddle structure.}
\end{equation}
With slight abuse of notation, we shall also call $(x_*, v_*,
\lambda_*)$ an index-1 saddle if $x_*$ is an index-1 saddle and $(v_*,
\lambda_*)$ the associated minimal eigenpair.

Given a dimer length $h$ and a vector $v \in S_1 := \{ u \in X \sep \|
u \| = 1 \}$, we define
\begin{align*}
  \E_h(x, v) &:= \smfrac12 \b( E(x+h v) + E(x - h v) \b) \quad
  \text{and} \\
  E_h(x) &:= \inf_{v \in S_1} \E_h(x, v).
\end{align*}
If $\#\arg\min_{v \in S_1} \E_h(x, v) = 1$, then we also define
\begin{displaymath}
  V(x) := \arg\min_{v \in S_1} \E_h(x, v) 
\end{displaymath}
and we can then write $E_h(x) = \E_h(x, V(x))$.

Finally, we observe that
\begin{align}
  \label{eq:err_DxEh}
\D_x \E_h(x, v) &= \smfrac12 \b( \D E(x+hv) + \D E(x-hv)\b) = \D E(x)
+ O(h^2), \\
  \label{eq:err_Dx2Eh}
\D_x^2 \E_h(x, v) &= \smfrac12 \b( \D^2 E(x+hv) + \D^2 E(x-hv)\b) = \D^2 E(x)
+ O(h^2), \\
  \D_v \E_h(x, v) &= \smfrac{h}{2} \b( \D E(x+hv) - \D E(x-hv) \b) = 
  h^2 \D^2  E(x) v + O(h^4) \quad \mbox{and} \label{eq:defn:errorH} \\
  \D^2_v \E_h(x, v) &= \smfrac{h^2}{2} \b( \D^2 E(x+hv) + \D^2 E(x-hv) \b) = 
  h^2 \D^2  E(x) + O(h^4), \label{eq:defn:dtwoh}, \\
  \label{eq:err_DxDvEh}
  \D_x \D_v \E_h(x, v) &= h^2 \D^3 E(x) \cdot v + O(h^4).
\end{align}
where we note that these errors are uniform whenever $x$ remains in a
bounded set. For future reference, we define the discrete Hessian
operator
\begin{equation}
  \label{eq:defn:Hh}
  H_h(x; v) := h^{-2} \D_v \E_h(x, v).
\end{equation}

\subsection{Two basic dimer variants}
\label{eq:dimer_variants}
We now make precise two basic variants of the dimer method. The first
algorithm is a variation of the original dimer method
\cite{HenkJons:jcp1999,OlsenKroes:jcp2004}, alternating steps in the
position ($x_k$) and direction ($s_k$) variables, but employs 
a modification proposed
by \cite{ZhangDu:sinum2012}. Indeed, the following algorithm can be
thought of as \cite{ZhangDu:sinum2012} with $\lambda$ ($h$ in our
case) taken to be constant instead of $h \to 0$ as $k \to \infty$.

\medskip

{\bf Algorithm 1}
\begin{itemize}
\item[(0)] Choose $x_0, v_0 \in X$ with $\|v_0\| = 1$, $h > 0$ and
  step lengths $(\alpha_k)_{k \in \N}, (\beta_k)_{k \in \N}$.
\item[(1)] For $n = 0, 1, 2, \dots $ do 
\item[(2)] \qquad $s_{k} := - (I-v_k \otimes v_k) h^{-2} \D_v \E_h(x_k, v_k)$
\item[(3)] \qquad $v_{k+1} :=  \cos(\| s_k \| \beta_k) v_k + \sin(\|
  s_k \| \beta_k) \frac{s_k}{\|s_k\|}$ 
\item[(4)] \qquad $x_{k+1} := x_k - \alpha_k (I-2 v_{k} \otimes v_{k}) \D_x
  \E_h(x_k, v_{k})$.
\end{itemize}

\medskip

Our second variant of the dimer method that we consider is closer in
spirit to the class of walking methods which employ the minimal
eigenpair. These include Rational Function Optimization (RFO)
\cite{Cerjan:jcp1981, Simons:jpc1983, Banerjee:jpc1985}, which uses
either an exact or approximate Hessian directly, or the Activation
Relaxation Technique nouveau (ART Nouveau)\cite{Marinica:prb2011,
  Machado-Charry:jcp2011, Mousseau:jamop2012}, which uses the Lanczos
method to find the minimal eigenvector. This modification of the dimer
method can also be motivated by observations in
\cite{OlsenKroes:jcp2004} that undertaking more accurate rotation
steps may lead to fewer iterations. As an idealised variant of this
idea we consider a dimer algorithm where, at each iteration, an exact
rotation $v$ is computed.


\medskip

{\bf Algorithm 2}
\begin{itemize}
\item[(0)] Choose $x_0 \in X, h > 0, (\alpha_k)_{k \in \N}$.
\item[(1)] For $k = 0, 1, 2, \dots$ do
\item[(2)] \qquad $v_k \in \arg\min_{\| v \| = 1} \E_h(x_k, v)$ 
\item[(3)] \qquad $x_{k+1} = x_k - \alpha_k (I-2v_k\otimes v_k) \D_x
  \E_h(x_k, v_k)$
\end{itemize}

\begin{remark}
  1. Algorithm 1 is clearly well-defined. Algorithm 2 is well-defined
  if ${\rm dim}(X) < \infty$, however, step (2) in Algorithm 2 is not
  necessarily well-defined in Hilbert space. We shall show in Theorem
  \ref{th:local_conv_dimer}(b) that this step is well-defined if the
  starting guess is close to a saddle point. In practice, the minimisation
  with respect to $v$ may only be performed to within a specified
  tolerance (see \S\ref{sec:dimer_linesearch}).

  2.  Both Algorithm 1 and Algorithm 2 may be rewritten such that a
  step in the position variable $x$ is performed by employing the
  gradient $\D E(x)$ instead of the averaged gradient $\D_x \E_h(x_k,
  v_k)$. For the sake of uniformity and simplicity of presentation we
  do not explicitly consider these as well.

  However, we note that (1) all our results can be extended to these
  variants, and (2) it seems to us that this has minor effects on the
  accuracy and efficiency of the algorithms, with the exception that
  it requires additional gradient evaluations. 

  Instead, it might be advisable to ``post-process'' the dimer
  Algorithms 1 and 2 using such a modified scheme. Namely, we shall
  prove that Algorithms 1 and 2 converge to a point $(x_h, v_h)$ that
  is $O(h^2)$ close to an index-1 saddle. Post-processing would then
  yield the exact saddle point.

  3. A natural variant of step (4) of Algorithm 1 is to replace it with
  \begin{displaymath}
    x_{k+1} := x_k - \alpha_k (I-2 v_{k+1} \otimes v_{k+1}) \D_x
  \E_h(x_k, v_{k+1}).
  \end{displaymath}
  We have observed that, in practise, this does not change the number
  of iterations required to reach a specified residual, but that it
  doubles the number of force (gradient) evaluations. Note that with the
  formulation we use, $\D_v \E_h(x_k, v_k) = \smfrac12(\D
  E(x_k+hv_k)-\D E(x_k-h v_k))$, and $\D_x \E_h(x_k, v_k) = \smfrac12
  (\D E(x_k+h v_k) + \D E(x_k+h v_k))$ and therefore only two force
  evaluations $\D E(x_k \pm h v_k)$ are required. The variant proposed
  in item 2. of the present remark would require three force
  evaluations in each step.
\end{remark}

\subsection{The dimer saddle}
Our first observation is that the dimer method (in both variants we
consider) approximates the Hessian by a finite difference and the
gradient by an average. Therefore, the dimer iterates ($x_k, v_k$)
with fixed dimer length $h$ {\em cannot} in general converge to a
saddle but only to a critical point $(x_h, v_h)$ near a saddle, satisfying 
\begin{equation}
  \label{eq:dimersaddle}
\D_x \E_h(x_h, v_h) = 0 \;\; \mbox{and} \;\; 
 (I-v_h\otimes v_h) \D_v \E_h(x_h, v_h) = 0. 
\end{equation}
The existence (and local uniqueness) of such
critical points is established in the following result.

\begin{proposition}
  \label{th:dimer_saddle}
  Let $(x_*, v_*, \lambda_*)$ be an index-1 saddle, then there exists
  $h_0 > 0$ such that, for all $h \leq h_0$, there exist $x_h, v_h \in
  X$, $\lambda_h \in \R$ and a constant $C$, such that
  \begin{equation}
    \label{eq:dimer_saddle_system}
    \begin{split}
\D_x \E_h(x_h, v_h) \equiv \hspace*{2mm}    
 \smfrac12 \b( \D E(x_h +h v_h) + \D E(x_h - h v_h)\b)  &= 0, \\
\smfrac1{h^2} \D_v\E_h(x_h, v_h) \equiv
      \smfrac1{2h} \b( \D E(x_h + h v_h) - \D E(x_h - h v_h) \b) &=
      \lambda_h v_h, \\
      \smfrac12 \| v_h \|^2 &= \smfrac12.
    \end{split}
  \end{equation}
  and moreover
  \begin{equation}
    \label{eq:dimer_saddle_error}
    \| x_h - x_* \| + \| v_h - v_*\| + |\lambda_h - \lambda_*| \leq C h^2.
  \end{equation}
\end{proposition}
\begin{proof}[Idea of proof]
  The result is a consequence of the inverse function theorem. Comparing
  \eqref{eq:dimer_saddle_system} with the exact saddle $(x_*, v_*,
  \lambda_*)$ a Taylor expansion shows that the residual is of order
  $O(h^2)$. Similarly, the linearisation can be shown to be $O(h^2)$
  close (in operator norm) to the linearisation of the exact saddle
  system $\D E(x_*) = 0, \D^2 E(x_*) v_* = \lambda_* v_*,
  \|v_*\|=1$. The linearisation of the latter is an isomorphism by the
  assumption that $x_*$ is an index-1 saddle. 
  The complete proof is given in \ref{sec:proof_dimer_saddle}.
\end{proof}


We shall refer to a triple $(x_h, v_h, \lambda_h) \in X \times
X \times \R$ that satisfies \eqref{eq:dimer_saddle_system} as a {\em dimer
  saddle}.

\subsection{Local convergence}
\label{sec:local_conv}
We now state local convergence results for the two dimer variants
formulated in Algorithm 1 and Algorithm 2. The main observation is
that Algorithm~1 need not converge monotonically, but that Algorithm 2
is in fact contractive.

\begin{theorem}
  \label{th:local_conv_dimer}
  Let $x_*$ be an index-1 saddle with minimal eigenpair $(\lambda_*,
  v_*)$. Then there exists a radius $r$, a maximal dimer length $h_0$
  and maximal step sizes $\bar\alpha$ and $\bar\beta$ 
 (independent of one another) as
  well as a dimer saddle $(x_h, v_h, \lambda_h)$ satisfying
  \eqref{eq:dimer_saddle_system} such that the following hold
  for all $h \leq h_0$:
  \begin{itemize}
  \item[(a)] Let $x_0 \in B_r(x_*), v_0 \in B_r(v_*), \sup_k \alpha_k
    \leq \bar\alpha, \sup \beta_k \leq \bar\alpha, \inf_k \alpha_k >
    0, \inf \beta_k > 0$, and let $(x_k, v_k)$ be the iterates
    generated by Algorithm 1, then there exist $C > 0, \eta \in (0,
    1)$ such that
    \begin{equation}
      \label{eq:convergence_dimer1}
      \| x_k - x_h \| + \| v_k - v_h \| \leq C \eta^k \b(
      \|x_0-x_h\|+\|v_0-v_h\| \b).
    \end{equation}

  \item[(b)] Let $x_0 \in B_r(x_*), h \leq h_0, \sup_k \alpha_k \leq
    \bar\alpha, \inf_k \alpha_k > 0$, then Algorithm 2 is well-defined
    (i.e., step (2) has a unique solution) and there exists $\eta \in
    (0, 1)$ such that
    \begin{equation}
      \| x_{k+1} - x_h \| \leq \eta \| x_k - x_h \| \qquad \text{for all
        $k \geq 0$.} \label{decrease}
    \end{equation}
    Moreover, there exists a constant $C$ such that $\|v_k - v_h\|\leq
    C \| x_k - x_h \|$.
  \end{itemize}
\end{theorem}
\begin{proof}[Idea of proof]
  (a) The proof of case (a) is a modification of the proofs of
  \cite[Thm. 2.1 and Thm. 3.1]{ZhangDu:sinum2012}. Upon linearisation
  of the updates about the exact saddle $(x_*, v_*)$, the updates can
  be re-written as

  \begin{align}
    \label{eq:prfa:mainstep}
    \left( \begin{matrix} x_{k+1}-x_h \\ v_{k+1}-v_h \end{matrix} \right)
    &= \left[ I - \left( \begin{matrix} \alpha_k A & 0 \\ \beta_k B &
          \beta_k C \end{matrix} \right) \right]
    \left( \begin{matrix} x_k - x_h \\ 
        v_k - v_h \end{matrix} \right)  + O\b( (\alpha_k + \beta_k) (h^2 + r_k)
    r_k \b),
    %
  \end{align}
  where $r_k^2=\|x_k-x_h\|^2+\|v_k-v_h\|^2$,
  \begin{equation}
    \label{eq:prfa:defn_ABC}
    A = (I-2 v_* \otimes
  v_*) \D^2 E(x_*), \qquad C = (I-v_*\otimes v_*) \D^2 E(x_*) -
  \lambda_* I,
  \end{equation}
  and $B$ is a bounded linear operator (the precise form is not
  important).


  Clearly, $A, C$ are both symmetric and positive definite, hence the
  spectrum of ${\bf A} = ( \alpha A, 0; \beta B, \beta C)$ is strictly
  positive. If we chose $\alpha_k \equiv \alpha, \beta_k \equiv \beta$
  constant, then \eqref{eq:convergence_dimer1} follows from standard
  stability results for dynamical systems. The (straightforward)
  generalisation, together with complete proof of
  \eqref{eq:prfa:mainstep} are given in \S\ref{sec:prfa}

  (b) We first note that step (2) of Algorithm 2 is well-defined due
  to the fact that $\D^2 E(x)$ has index-1 structure for all $x \in
  B_r(x_*)$, if $r$ is chosen sufficiently small. In this case an
  implicit function argument guarantees the existence of a unique
  solution $v_k = V(x_k)$. This is made precise in Lemma
  \ref{th:aux_lemma_Eh}. 

  In the same lemma we also show that $\D^2 E_h(x) = \D^2 E(x)+
  O(h^2)$ for all $x \in B_r(x_*)$. This allows us to linearize step
  (3) in Algorithm 2 to obtain
  \begin{displaymath}
    x_{k+1} - x_h  = \b(I - \alpha_k A
    \b)(x_k - x_h) + O(
    r_k^2 + h^2 + \alpha_k^2 ) r_k,
  \end{displaymath}
  where, again, $A = (I-2 v_* \otimes v_*) \D^2 E(x_*)$ and $r_k =
  \|x_k-x_h\|$.  Since $A$ is positive definite the result follows
  easily. The complete proof is given in \S\ref{sec:prfb}.
\end{proof}

\section{A Dimer Algorithm with Linesearch}

\subsection{Motivation: a local merit function}
\label{sec:ls_motivation}
Let $x_* \in X$ be an index-1 saddle with minimal eigenpair $(v_*,
\lambda_*)$, and consider the modified energy functional
\begin{displaymath}
  F(x) := E(x) + \frac{\kappa}{2} \b( v_* \cdot (x - x_*) \b)^2.
\end{displaymath}
Then, $\D F(x_*) = 0$ and $\D^2 F(x_*) = (I+\kappa v_*\otimes v_*) \D^2
E(x_*)$, which is positive definite if and only if $\kappa > -
\lambda_*$. For this choice, it follows that $x_*$ is a strict local
minimizer of $F$.

The dimer variant of this observation is that, if $(x_h, v_h,
\lambda_h)$ is a dimer saddle point (cf. Theorem
\ref{th:dimer_saddle}) and we define a modified energy functional
\begin{displaymath}
  F_h(x) := \E_h(x, v_h) + \frac{\kappa}{2} \b( v_h \cdot (x - x_h) \b)^2,
\end{displaymath}
then choosing $\kappa > - \lambda_*$ and $h$ sufficiently small again
guarantees that $x_h$ becomes a local minimizer of $F_h$. We can
make this precise (and generalise) as follows.

\begin{lemma}
  \label{eq:positivity_Fh}
  Let $x_0 \in X$ such that $\D^2 E(x_0)$ has index-1 saddle structure
  with minimal eigenpair $(V(x_0), \lambda)$ and $\mu > 0$ such that $y
  \cdot (\D^2 E(x_0) y) \geq \mu \|y\|^2$ for $y \in
  \{V(x_0)\}^\perp$. Fix $r, h_0 > 0$.

  Let $0 < h \leq h_0, v_0 \in X, \|v_0 \|=1, g_0 \in X$ and
  \begin{displaymath}
    F_0(x) := \E_h(x, v_0) + g_0 \cdot (x-x_0) + \smfrac{\kappa}{2} \b( v_0 \cdot (x - x_0) \b)^2,
  \end{displaymath}
  then there exists $C = C(x_0, r, h_0)$ such that, for all $x \in
  B_r(x_0), h < h_0, y \in X$,
  \begin{displaymath}
    y \cdot (\D^2 F_0(x) y) \geq \B( \min\b( \mu, \kappa+\lambda \b) - C \b(h^2+\|v_0-V(x_0)\| + \|x - x_0 \|
    \b) \B) \| y \|^2.
  \end{displaymath}
\end{lemma}
\begin{proof}
  For $x = x_0$, we compute $\D^2 F_0(x_0) = \D^2 E(x_0) + O(h^2) +
  \kappa v_0 \otimes v_0$. Then, the result follows readily from the
  observation that
  \begin{align*}
    (v_0 \cdot y)^2
    &= (v \cdot y)^2 + ((v_0-v)\cdot y)
    ((v_0+v)\cdot y)  \\
    &\geq (v \cdot y)^2 - 2 \| v_0 - v \| \|y\|^2.
  \end{align*}

  For general $x$, the result follows from local Lipschitz continuity
  of $\D^2 E$.
\end{proof}

To complete the definition of $F_0$ we must specify $g_0, \kappa$. The
strategy is to choose it in such a way that minimising $F_0$ will lead
to an improved approximation for $x$. 

From the inverse function theorem it follows that there exists
$\tilde{x} = x_* + O(h^2) = x_h + O(h^2)$ such that $\D_x
\E_h(\tilde{x}, v_0) = 0$ (we will make this precise below), and Lemma
\ref{eq:positivity_Fh} allows us to assume that it is in fact a local
minimiser of $F_0$. When minimising $F_0$, we therefore hope to obtain
a point ``close to'' $\tilde{x}$. To test this, we evaluate the
residual at $x = \tilde{x}$,
\begin{align*}
  \D F_0(\tilde{x}) &= \D_x \E_h(\tilde{x}, v_0) + g_0 + \kappa (v_0
  \otimes v_0)  (\tilde{x}-x_0) \\
  &\approx g_0 + \kappa (v_0 \otimes v_0) \D_x^2 \E_h(x_0, v_0)^{-1}
  \b(\D_x \E_h(\tilde{x},v_0)-\D_x \E_h(x_0, v_0) \b) \\
  &\approx g_0 - \smfrac{\kappa}{\lambda_0} (v_0 \otimes v_0) 
  \D_x \E_h(x_0, v_0),
\end{align*}
where $\lambda_0 = H_h(x_0; v_0)\cdot v_0$. This leads to the choice
\begin{displaymath}
  g_0 := \frac{\kappa}{\lambda_0} (v_0 \otimes v_0) \D_x \E_h(x_0, v_0).
\end{displaymath}
Note in particular, that the steepest descent direction for $F_0$ at
$x_0$ is 
\begin{displaymath}
  -\D F_0(x_0) = \b(I + \smfrac{\kappa}{\lambda_0} (v_0 \otimes v_0) \b) \D_x \E_h(x_0, v_0).
\end{displaymath}
For the special choice $\kappa = - 2 \lambda_0$, this yields the
standard dimer search direction.

\subsection{Dimer algorithm with linesearch}
\label{sec:dimer_linesearch}
Given an iterate $x_k$, $v_k$ and $\lambda_k := v_k \cdot H_h(x_k;
v_k)$, we define the auxiliary functional $F_k \in C^4(X)$,
\begin{align}
  \label{eq:defn_aux_fcnl}
  F_k(x) &:= \E_h(x, v_k) - 2 \b[(v_k \otimes v_k) \D_x \E_h(x_k, v_k)\b] \cdot
  (x-x_k) -
  \lambda_k \b\| (v_k\otimes v_k) (x-x_k)\b\|^2 \\
  \notag
  &\,\,= \E_h(x, v_k) - 2 \b((v_k \cdot \D_x \E_h(x_k, v_k)\b) \cdot
  \b( v_k \cdot (x-x_k)\b) -
  \lambda_k \b(v_k \cdot (x-x_k)\b)^2,
\end{align}
motivated by the discussion in \S\ref{sec:ls_motivation}. Instead of
locally minimising $F_k$ we only perform a minimisation step in the
steepest descent direction, using a standard linesearch procedure
augmented with the following sanity check: For a trial $x^{\rm t} =
x_k - \alpha \D F_k(x_k)$ we require that $v_k$ is still a reasonable
dimer orientation for $x^{\rm t}$ by checking the residual
$\| (I-v_k\otimes v_k) H_h(x^{\rm t}; v_k) \|$. If this
residual falls above a certain tolerance then we reject the step and
reduce the step size.

\medskip

{\bf Algorithm 3: }
\begin{enumerate}
\item {\bf Input:} $x_0, v_{-1}, h$ \\
  {\bf Parameters: } $\beta_{-1}, \alpha_0, \alpha_{\rm max} > 0, \Theta \in
  (0, 1), \Psi > 1$

\medskip
\item {\bf For} $k = 0, 1, 2, \dots$ {\bf do} 
\medskip
\item[] \qquad {\it \%\% Rotation \%\%}
\item \qquad $[v_k, \beta_k] := {\bf Rotation}[x_k, v_{k-1}, \beta_{k-1}]$

\medskip

\item[] \qquad {\it \%\% Translation \%\%}
\item \qquad $p := - \D F_k(x_k)$
\item \qquad $\alpha := \min(\alpha_{\rm max}, 2 \alpha_{k-1})$
\item \qquad {\bf While } ($F_k(x_k+\alpha p) > F_k(x_k) - \Theta
  \alpha \|p\|^2$)  \\
  {\color{white} .} \hspace{1.2cm} {\bf or }  ($\| (I-v_k\otimes v_k) H_h(x_k+\alpha
  p; v_k) \| > 
\Psi \|\D_x \E_h(x_k, v_{k-1}) \|
$) { \bf do}
\item \qquad \qquad $\alpha := \alpha / 2$
\item \qquad $x_{k+1} := x_k + \alpha p$; $\alpha_k := \alpha$
\end{enumerate}

\medskip

It remains to specify step (3) of Algorithm 3. Any method computing an
update $v_k$ satisfying $\| (I-v_k\otimes v_k) H_h(x_k;v_k)\| \leq
{\rm TOL}$, for given {\rm TOL}, is suitable; we
prescribe the tolerance ${\rm TOL} = \|\D_x \E_h(x_k, v_{k-1}) \|$
so long as this isn't too large.
A basic choice of method is the
following projected steepest descent algorithm.

\medskip

{\bf Rotation: }
\begin{enumerate}
\item {\bf Input: } $x, v, \beta$ \\
  {\bf Parameters: } ${\rm TOL} = \min(\| \D_x \E_h(x, v) \|, {\rm
    TOL}_v^{\rm hi})$, $\beta_{\rm max}$, $\Theta$

\medskip
\item {\bf While } $\| (I-v \otimes v) H_h(x; v) \| >
  {\rm TOL}$ {\bf do } 
\item \qquad $s := -(I-v \otimes v) H_h(x; v)$
\item \qquad $r := \|s\|; \quad \beta := \min(\beta_{\rm max}, 2\beta)$
\item \qquad $v_\beta := \cos(\beta r) v + \sin(\beta r) \smfrac{s}{r}$
\item \qquad {\bf While } $\E_h(x, v_\beta) > \E_h(x,
  v) - \Theta \beta \|s\|^2$ {\bf do }
\item \qquad \qquad $\beta := \beta / 2$
\item \qquad $v := v_\beta$

\medskip
\item  {\bf Output: } $v, \beta$

\end{enumerate}

\medskip

%

\begin{proposition}
  Algorithm 3 is well-defined in that the rotation step (3) as well as
  the linesearch loop (6, 7) both terminate after a finite number of
  iterations, the latter provided that $\D_x \E_h(x_k, v_k) \neq
  0$. 
\end{proposition}
\begin{proof}
  The Rotation Algorithm employed in step (3) of Algorithm 3
  terminates for any starting guess due to the fact that it is a
  steepest descent algorithm on a Stiefel manifold (the unit sphere)
  with a backtracking linesearch employing the Armijo condition
\cite{MikhRedkoPere87}.  Convergence of this iteration to a critical
  point is well known \cite[Chap.4]{AbsiMahoSepu08}. 
  %
  The loop (6,7) terminates after a finite number of iterations 
  \cite{NocedalWright} since
  $p$ is a descent direction for $F_k \in C^4$, that is, $F_k(x_k+\alpha p) =
  F_k(x_k) - \alpha \|p\|^2 + O(\alpha^2)$.
\end{proof}

\begin{remark}
  1. In practise, the algorithm terminates, once the entire dimer
  saddle residual reaches a prescribed tolerance, i.e., $\|\D_x
  \E_h(x_k, v_k)\| \leq {\rm TOL}^x$ in addition to $\|(I-v_k\otimes
  v_k) H_h(x_k; v_k) \| \leq \|\D_x \E_h(x_k, v_k)\|$.


  
  2. The two basic backtracking linesearch loops (5)--(8) and
  (11)--(12) can (and should) be replaced with more effective
  linesearch routines in practise, in particular choosing more
  effective starting guesses and using polynomial interpolation to
  compute linesearch steps. However, the discussion in
  \S\ref{sec:doublewell} indicates that a Wolfe-type termination
  criterion might be inappropriate.
\end{remark}

\subsection{Failure of global convergence}
\label{sec:doublewell}
The modifications of the original dimer algorithms that we have in
Algorithm 3 would, in the case of optimisation, yield a globally
convergent scheme. Unfortunately, this is not the case in the saddle
search case. To see this, consider a one-dimensional double-well example,
\begin{equation}
  \label{eq:doublewell}
  E(x) = \smfrac14 (1-x^2)^2 = \smfrac14 x^4 - \smfrac12 x^2 + \smfrac14;
\end{equation}
cf. Figure~\ref{fig:doublewell}(a).  There are only two possible
(equivalent) dimer orientation $v = \pm 1$, and therefore the rotation
steps in Algorithm 3 are ignored. We always take $v = 1$ without loss
of generality. The translation search direction at step $k$ is always
given by $p = -(1-2) \D_x \E_h(x_k, 1) = \D_x \E_h(x_k, 1)$, i.e., an
ascent direction.

\begin{figure}
  \includegraphics[height=5cm]{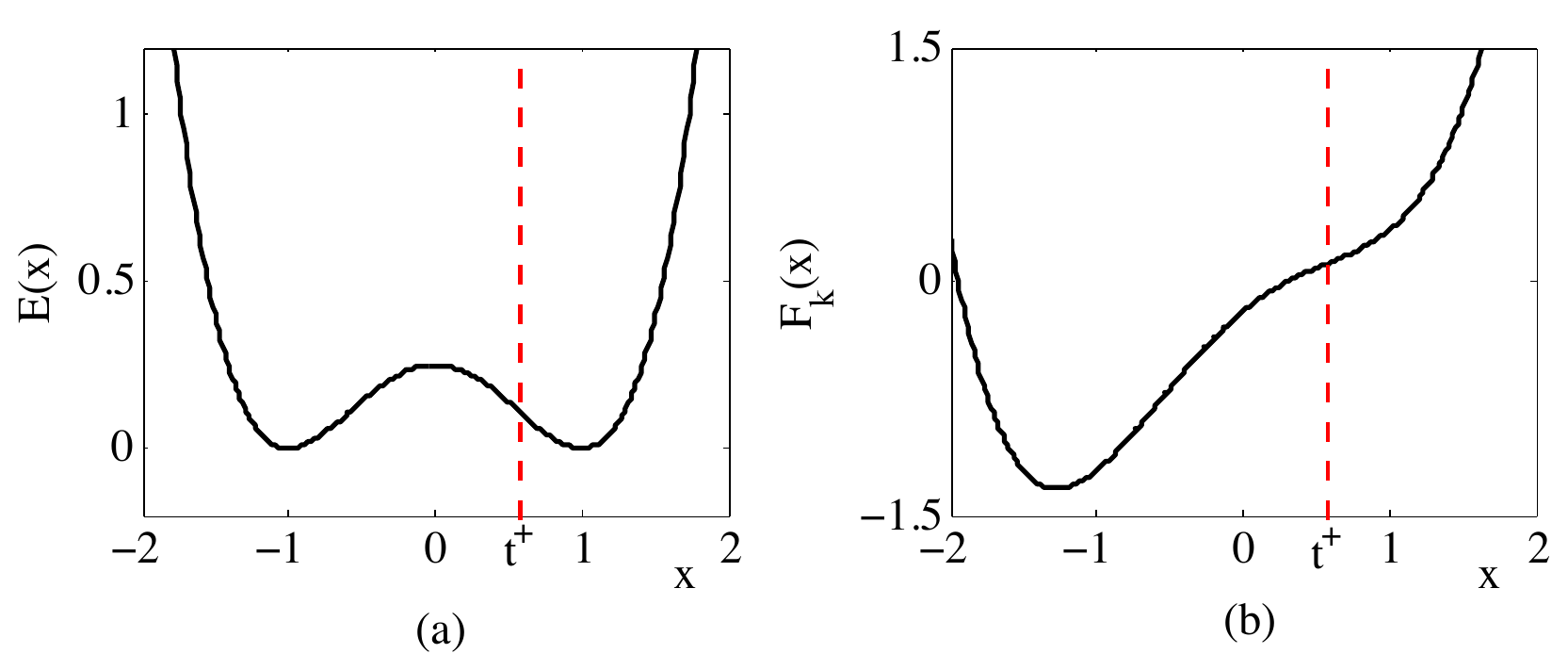}
  \caption{ \label{fig:doublewell} (a) Double-well energy defined in
    \eqref{eq:doublewell}. (b) The auxiliary functional $F_k(x)$ with
    $x_k = t^+_h$; cf. \S\ref{sec:doublewell}. The second turning
    point $t^-_h = - t^+_h$ is an admissible descent step for $F_k$,
    hence the dimer method can potentially cycle.}
\end{figure}

It is easy to see that $x_* = 0$ is an index-1 saddle (i.e., a
maximum), and that there are two turning points $t^\pm = \pm
3^{-1/2}$. Thus, there exist ``discrete turning points'' $t^\pm_h =
\pm 3^{-1/2} + O(h^2)$ such that $\lambda(t^\pm_h) = 0$, where
$\lambda(x) = H_h(x; 1) \cdot 1 = \smfrac{1}{2 h^{2}}
(E'(x+h)-E'(x-h))$.

Suppose that we have an iterate $x_k = t^+_h$, then the translation search
direction is $p^+ = \D_x \E_h(t^+_h, 1) < 0$. Since $\E_h(t^-_h) =
\E_h(t^+_h)$ it follows that 
\begin{displaymath}
  F_k(t^-_h) = \E_h(t^-_h) - 2 p^+ (t^-_h-t^+_h) < \E_h(t^-_h) = F_k(t^+_h).
\end{displaymath}
Thus, for $\Theta$ sufficiently small, the update $x_{k+1} = t^-_h$
satisfies all the conditions for termination of the loop (11)--(12) in
Algorithm 3. See also Figure~\ref{fig:doublewell} (b), where $F_k$ is
visualised.

We therefore conclude that our newly proposed variant of the dimer
algorithm does not excluded cycling behaviour. We also remark that the
example is not exclusively one-dimensional, but that analogous
constructions can be readily made in any dimension.

\subsection{Local convergence}
We now establish a local convergence rate.

\begin{theorem}
  \label{th:conv_ls}
  Let $(x_*, v_*, \lambda_*)$ be an index-1 saddle, let $(x_h,
    v_h, \lambda_h)$ denote the dimer saddle associated with $(x_*,
    v_*, \lambda_*)$ (cf. Theorem \ref{th:dimer_saddle}) and let $x_k,
    v_k$ be the iterates generated by the Linesearch Dimer Algorithm.
    Then there exist $r, h_0, C > 0$ and $\gamma \in (0, 1)$ such
    that, for $x_0 \in B_r(x_*)$, $v_{-1} \in B_r(v_*) \cap S_X$ and
    $h \leq h_0$, one of the following alternatives are true:
  \begin{itemize}
  \item[(i)] If $\D_x \E_h(x_k, v_{k-1}) = 0$ for some $k \in \N$,
    then $\| x_k - x_h \| \leq C h^2$.
  \item[(ii)] If $\D_x \E_h(x_k, v_{k-1}) \neq 0$ for all $k \in \N$,
    then 
    \begin{equation}
      \label{eq:conv_ls_maineq}
      \| x_k - x_h \| + \|v_k - v_h \| \leq C \gamma^k \b(\|x_0 - x_h\| 
      + h^2 \| v_{-1}-v_h\| \b),
    \end{equation}
  \end{itemize}
\end{theorem}
\begin{proof}[Sketch of proof]
   Case (i) merely serves to exclude an unlikely situation, in
    which the Rotation algorithm is ill-defined. We do not discuss
    this case here, but treat it in \S\ref{sec:prf_case_i}. In the
    following assume Case (ii).

	Let $r_k = \| x_k - x_h \|$ and $s_k := \|v_k - v_h \|$.

  0. We recall basic contraction results for Armijo-based linesearch 
  methods both in a general Hilbert space and for iterates constrained
  to lie on the unit sphere in \S\ref{sec:analysis_sd}.

  1. As a first proper step we establish that, under the termination
  criterion $\|(1-v_k\otimes v_k) H_h(x_k; v_k)\| \leq \| \D_x
  \E_h(x_k, v_{k-1}) \|$ for the rotation step, it follows that $\|v_k
  - v_h \| \lesssim r_k + h^2 s_{k-1}$.  This is proven in Lemma
  \ref{th:prfls:L1} and Lemma \ref{th:prfls:rotation_lemma}.

  2. Next, we use this result to establish that there exists a local
  minimizer $y_k$ of $F_k$ satisfying $\| y_k - x_h \| \lesssim r_k^2
  + h^2r_k + h^4 s_{k-1}$. This is established in Lemma \ref{th:prfls:L2}.

  3. The linesearch procedure and the upper bound on the step length
  ensure that the step of $x_k$ to $x_{k+1}$ contracts towards
  $y_k$, that is, $\| x_{k+1} - y_k \|_* \leq \gamma_* \| x_k - y_k
  \|_*$ for some $\gamma_* \in (0, 1)$ and $\|\cdot\|_*$ the {\em energy} norm
  induced by $(I-2v_*\otimes v_*) \D^2 E(x_*) \approx \D^2 F_k(y_k)$.
  This is obtained in Lemma \ref{th:prfls:contr_x}.

  4. The three preceding steps can then be combined to establish that, for
  $r_0, s_{-1}, h$ sufficiently small, there exists a constant
  $\gamma_3 \in (\gamma_*, 1)$ such that
  \begin{displaymath}
    r_{k+1}^* + h^2 s_k \leq \gamma_3 (r_k^* + h^2 s_{k-1}),
  \end{displaymath}
  where $r_k^*: = \|x_k - x_h\|_*$. This contraction result readily
  implies the result of the theorem.

  The complete proof is given in \S\ref{sec:prfls}.
\end{proof}

\section{Numerical Tests}

\subsection{Remarks on the implementation}
\label{sec:implementation}
Here, we remark on how preconditioning is implemented and on some
further details of our implementation that slightly deviate from the
theoretical formulations of Algorithms 1 and 3.

In all cases the underlying space is $X = \R^N$ for some $N \in
\N$. The main deviation from Algorithms 1 and 3 is that we admit
general Euclidean norms and inner products that may change from one
step to another,
\begin{displaymath}
  \| u \| = \sqrt{u^T M_k u}, \qquad \text{and} \qquad u \cdot v = u^T M_k v,
\end{displaymath}
where $M_k$ is symmetric and positive definite. That is, our
implementation is a {\em variable metric variant}. 

Let $E \in C^4(X) = C^4(\R^N)$, and let $\D'$ denote the standard
gradient and $\otimes'$ the standard tensor product (i.e., the
gradient and tensor products with respect to the $\ell^2$-norm), then
the gradient and tensor products in step $k$ become
\begin{displaymath}
  \D E(x) = M_k^{-1} \D' E(x), \quad \text{and} \quad (v \otimes v) \D
  E(x) = (v\otimes' v) \D' E(x).
\end{displaymath}
The variable metric variant of Algorithm 1, augmented with a
termination criterion, is given below. For the purposes of the
numerical testing we call this the {\em simple dimer} method, it is
effectively a forward Euler ODE integrator for the dimer
dynamics. (Note also that here the rotation step is performed by a
simple descent step followed by a projection, rather than a step on
the manifold.)

\medskip

{\bf Algorithm 1$^{\mbox{\protect\tiny vm}}$:}
\begin{enumerate}

\item {\bf Input: } $x_0, v_0 \in X$, $h>0$,$\alpha,\beta>0$,${\rm TOL}^x,{\rm TOL}^v > 0$; $k := 0$;
\item {\bf While } $ \| M_k^{-1/2} \D_x' \E_h(x_k, v_k)
  \|_{\ell^2} >  {\rm TOL}^x $ \\
  {\color{white} .} \hspace{2.1cm}  {\bf or } $ \| (M_k^{-1/2} - M_k^{1/2} v_k\otimes' v_k) h^{-2} \D_v'
  \E_h(x_k, v_k) \|_{\ell^2}  > {\rm TOL}^v$  \,\,\,{\bf do}
\item[] \qquad {\it \%\% Metric \%\%}
\item \qquad Compute a spd matrix $M_k \in \R^{N \times N}$;
\item \qquad $v_{k} := v_{k} / \| M_k^{1/2} v_{k} \|$;
\item\qquad $v_{k+1} := v_{k} - \beta (M_{k}^{-1}-v_k \otimes v_k) h^{-2} \D_v'
  \E_h(x_k, v_k)$
\item\qquad $x_{k+1} := x_k - \alpha (M_{k}^{-1}-2 v_{k} \otimes v_{k}) \D_x'
  \E_h(x_k, v_{k})$.
\item \qquad $k := k+1$
\end{enumerate}
\medskip
\begin{remark}
  In our experiments we observe that the rotation residual decreases
  more quickly than the translation residual, hence the convergence
  criteria could be based on the translation residual only, without
  affecting the results.
\end{remark}

\begin{remark}
  \label{rem:vm}
  Our analysis of both the Simple Dimer Algorithm and of the
  Linesearch Dimer Algorithm is readily extended to their variable
  metric variants, provided that the metric $M_k$ at iterate $k$ is a
  smooth function of the state, i.e., $M_k = {\bf M}(x_k, v_k)$, where
  ${\bf M} \in C^2(B_r(x_*) \times S_X; L(X))$, for some $r >
  0$. This is the case in all examples that we consider below. A more
  general convergence theory, e.g., employing quasi-Newton type
  hessian updates requires additional work.
\end{remark}
\medskip

Analogous modifications are made to Algorithm 3.  The auxiliary functional $F_k$ now reads
\begin{align*}
  F_k(x) &= \E_h(x; v_k) - 2 \b(v_k^T
  \D_x' \E_h(x_k, v_k)\b) \b( v_k^T M_k (x - x_k) \b)
  + \lambda_k \b( v_k^T M_k (x - x_k) \b)^2, \\
  \lambda_k &= h^{-2} v_k^T \D_v'\E_h(x_k, v_k), \\
  \D_x' \E_h(x, v) &= \smfrac{1}{2} \b( \D' E(x+hv) + \D' E(x - hv)
  \b), \\
  \D_v' \E_h(x, v) &= \smfrac{h}{2} \b( \D' E(x+hv) - \D' E(x-hv) \b),
\end{align*}
where we recall that $\D'$ denotes the standard gradient (i.e., the
gradient with respect to the $\ell^2$-norm).


\medskip \noindent 
{\bf Algorithm 3$^{\mbox{\protect\tiny vm}}$:}
\begin{enumerate}
\item {\bf Input: } $x_0, v_0 \in X$, $h>0$,${\rm TOL}^x,{\rm TOL}^v > 0$; $k := 0$;


\medskip
\item {\bf While } $ \| M_k^{-1/2} \D_x' \E_h(x_k, v_k)
  \|_{\ell^2} >  {\rm TOL}^x$
\medskip
\item[] \qquad {\it \%\% Metric \%\%}
\item \qquad Compute a spd matrix $M_k \in \R^{N \times N}$;
\item \qquad $v_{k}' := v_{k} / \| M_k^{1/2} v_{k-1} \|$;
\medskip
\item[] \qquad {\it \%\% Rotation \%\%}
\item \qquad $v_{k+1} := \text{\bf Rotation (VM)}[x_k, v_{k}', \beta, M_k]$

\medskip

\item[] \qquad {\it \%\% Translation \%\%}
\item \qquad $p_M := -(M_k^{-1}-2v_{k+1} \otimes v_{k+1})  \D_x' \E_h(x_k; v_{k+1})$
\item \qquad $\alpha := \min(\alpha_{\rm max}, 2 \alpha)$
\item \qquad {\bf While } ($F_k(x_k+\alpha p_M) > F_k(x_k) - \Theta
  \alpha p_M^T M_k p_M$)  \\
  {\color{white} .} \hspace{1.2cm} {\bf or }  ($\| M_k^{1/2} (M_k^{-1}-v_{k+1}\otimes v_{k+1})
  h^{-2} \D_v' \E_h(x_k+\alpha p_M; v_{k+1}) \|_{\ell^2} > $  \\  {\color{white} .} \hspace{2.0cm} $\Psi \| M_k^{1/2} (M_k^{-1}-v_{k+1}\otimes v_{k+1})
  h^{-2} \D_v' \E_h(x_k; v_{k+1}) \|_{\ell^2}$) { \bf do}
\item \qquad \qquad $\alpha := \alpha / 2$
\item \qquad $x_{k+1} := x_k + \alpha p_M$.
\item \qquad $k := k+1$
\end{enumerate}
\medskip

\medskip \noindent
{\bf Rotation$^{\mbox{\protect\tiny vm}}$:}
\begin{enumerate}
\item {\bf Input: } $x, v, \beta, M_k$ \\
  {\bf Parameters: } ${\rm TOL} = \max(\| M_k^{-1/2} \D_x' \E_h(x, v) \|_{\ell^2}, {\rm
    TOL}^v$), $\Theta \in (0, 1)$, $\beta_{\rm max}$;

\medskip
\item {\bf While } $\| M_k^{1/2} (M_k^{-1}-v \otimes v) h^{-2} \D_v' \E_h(x; v) \| >
  {\rm TOL}$ {\bf do } 
\item \qquad $s := -(M_k^{-1}-v \otimes v)  h^{-2} \D_v' \E_h(x; v)$
\item \qquad $t := \| M_k^{1/2} s \|_{\ell^2}$; $\beta :=
  \min(\beta_{\rm max}, 2\beta)$
\item \qquad $v_\beta := \cos( t \beta ) v + \sin(t\beta)
  t^{-1} s$
\item \qquad {\bf While } $\E_h(x, v_\beta) > \E_h(x,
  v) - \Theta \beta t^2$ {\bf do }
\item \qquad \qquad $\beta := \beta / 2$
\item \qquad $v := v_\beta$

\item  {\bf Output: } $v, \beta$ 
\end{enumerate}

\medskip
\begin{remark}
	An additional (optional) modification that can give significant
  performance gains is to employ a different heuristic for the initial
  guess of $\alpha$ in Step (7) of  Algorithm 3$^{\mbox{\protect\tiny vm}}$: With $p_{M,k}
  := -(M_k^{-1}-2v_k \otimes v_k) \D_x' \E_h(x_k; v_k)$ and $p_{I,k}
  := - (I - 2 v_k \otimes v_k) \D_x' \E_h(x_k, v_k)$ let, for $k \geq
  2$, $\gamma_k := (p_{M,k-1} \cdot' p_{I,k-1})/(p_{M,k}\cdot'
  p_{I,k})$, then for $k \geq 2$ we replace Step (7) with 
  \begin{displaymath}
    \alpha :=
    \min\b( \mathrm{avg}(\gamma_{\max(2,k-4)},
    \ldots, \gamma_k), 2 \alpha, \alpha_{\rm max})
  \end{displaymath}

  An analogous modification can be made for the rotation algorithm.
\end{remark}

In all numerical tests we use the following parameters: 
$h=10^{-3}$, $\Theta = \sqrt{0.1}$, ${\rm TOL}^x = 10^{-5}$, 
${\rm TOL}^v = 10^{-1}$,  $\alpha_{\rm max} = 1$ 
and $\Psi= 100$. We briefly discuss these choices:
\begin{itemize} 
  \item $h$ should be small enough such that the dimer saddle is
    sufficiently close to the true saddle (with respect to the length
    scales of the given problem), while large enough that numerical
    robustness does not become a problem for the rotation. In all our
    tests, $h = 10^{-3}$ was a good compromise.

  \item $\Theta$ should be sufficiently large (though, $\leq 1/2$) to
    ensure that the linesearch method finds steps which give a large
    decrease in dimer energy. It is often chosen much smaller than our
    choice of $\Theta = \sqrt{0.1}$ to immediately accept steps that
    make some progress. Our experience is that, with preconditioned
    search direction, our more stringent choice gives better
    performance.

  \item The choice of ${\rm TOL}^x$ simply controls the desired level
    of convergence to the dimer saddle. 
    
  \item The parameter ${\rm TOL}^v$ should be chosen as weakly as
possible such that either algorithm converges to the saddle. In Algorithm 3$^{\mbox{\protect\tiny vm}}$ rotations
are performed such that the rotation residual is at least as good as the
translation residual until it moves below this value. Subsequent
translations may increase the rotation residual such that further
applications of the rotation algorithm are needed. In practise this
means that the rotation algorithm is performed at every iteration of
Algorithm 3' for the first few steps, then only sporadically or not at
all once the rotation residual reaches ${\rm TOL}^v$. The use
of this parameter then decreases the overall number of gradient
evaluations needed to find the dimer saddle, by only performing the
rotation as necessary.

  \item The maximum step $\alpha_{\rm max}$ should principally be
    chosen such that the dimer cannot translate into non-physical
    regimes for the given problem. 
    
    \item The parameter $\Psi$ should be chosen $>1$ and restricts
the translation step from moving the dimer to a point where it becomes
too badly orientated. In our numerical tests this parameter is set
sufficiently large that this termination criteria for the translation
never occurs (the translation always terminates by finding a sufficient
decrease in the auxiliary functional $F_k$).  

\end{itemize}

\begin{remark}
	We observe during numerical testing that the rotation component
  of the linesearch dimer is somewhat vulnerable to rounding error in
  the objective function $E$. As the dimer becomes increasingly well
  orientated, $\nabla E$ becomes almost orthogonal to the dimer
  orientation and any small rotation may result in a zero change (to
  numerical precision) in the dimer energy.   
  In the numerical examples presented in this section, this never occurs 
  since we use a relatively high value for ${\rm TOL}^v$, that is the 
  rotation is only ever weakly converged. In our examples this is sufficient 
  for the the dimer to converge to the saddle.  If a stronger level of
  converge were required, another technique should be used to improve the
  rotation residual further,
  such as changing to a gradient based method or simply making fixed
  steps.
\end{remark}

\subsection{Test 1: A simple 2D example}
Our first example is taken from \cite{ZhangDu:sinum2012}. We equip $X
= \R^2$ with the standard Euclidean inner product. The energy function
is given by $E(x,y) = (x^2 - 1)^2 + y^2$, which has two simple
symmetric minima at $(\pm 1, 0)$ and a unique index-1 saddle at
$(0,0)$. The energy function is given graphically in Figure \ref{fig:Prob1v2Fig1}.

Figure~\ref{fig:Prob1v2Fig2} shows the
$x$-residual $\| \D_x \E_h(x, v) \|$ plotted against the number of
function evaluations and the number of iterations.

The performance of the linesearch dimer is compared with a simple dimer
method with different step sizes. Evidently a good choice of step is
important. If a poor choice is made the algorithm may perform poorly or
diverge. The linesearch dimer method requires a certain amount of
overhead versus a simple dimer with well chosen step sizes. We can see
in Figure \ref{fig:Prob1v2Fig2} that the
linesearch dimer may find a solution in fewer dimer iterations than the
best fixed step tested (indicating that it found better steps), but
using more gradient evaluations.

\begin{figure}
\includegraphics[width=0.6\textwidth]{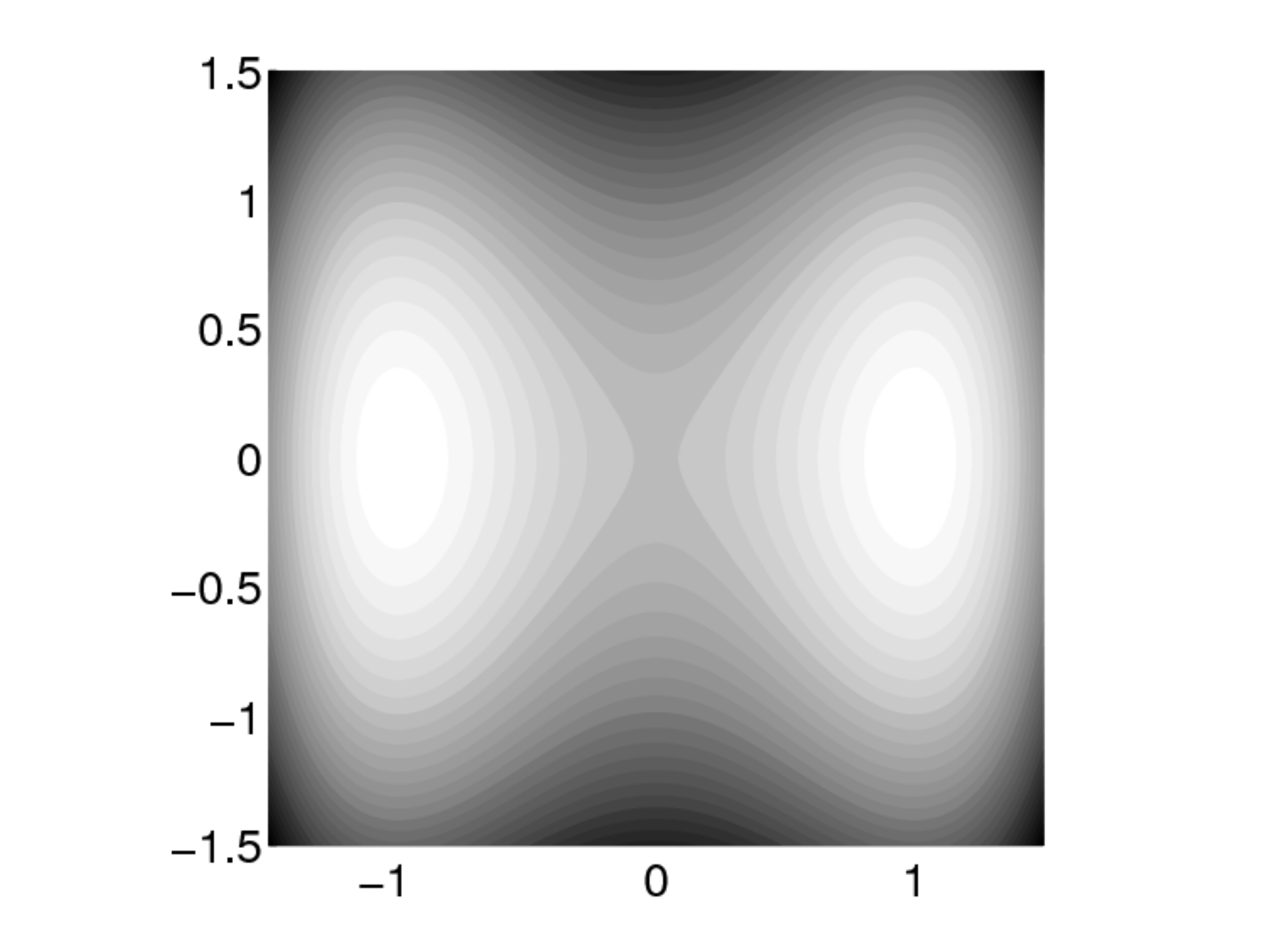}
\caption{ \label{fig:Prob1v2Fig1} Energy function for Test 1 with 2 symmetric minima and a unique index-1 saddle}
\end{figure}

\begin{figure}
\includegraphics[width=\textwidth]{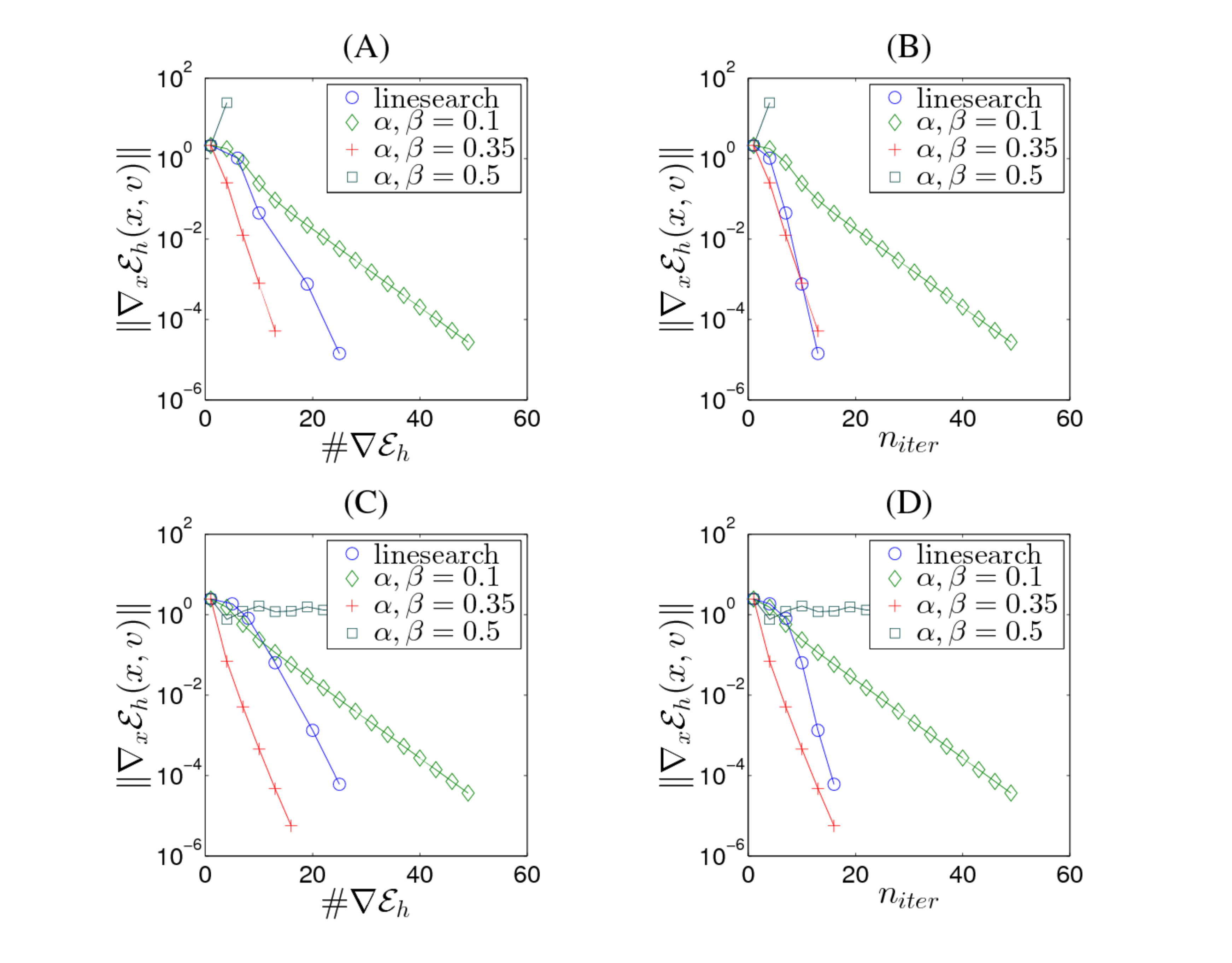}
\caption{ \label{fig:Prob1v2Fig2} Convergence of the dimer to the
    saddle in a simple 2D example (Test 1).(A,B) The $\ell_2$ norm of the $x$ gradient
    versus the number of force evaluations and the number of dimer iteration where the
   initial dimer state is $x=[0.2,1] ,v=[1,1] $. In this case the
    choice $\alpha,\beta = 0.5$ diverges immediately. (C,D) The $\ell_2$ norm of the $x$ gradient
    versus the number of force evaluations and the number of dimer iterations where the
    initial dimer state is $x=[0.2,1] ,v=[1,1] $.}
\end{figure}
%
%

\subsection{Test 2: Vacancy Diffusion}
Our second test case is a standard example from molecular
physics. A single atom is removed from a 2D lattice and a
neighbouring atom is moved partway into the gap. Atoms within a
certain radius of the vacancy are allowed to move, while those beyond
that radius are fixed. This configuration is illustrated in Figure
\ref{fig:Prob2Fig1}(A).

The energy function is given by the simple Morse potential,
\begin{equation}
E(\{x_i\}) = \sum_{i,j} V(\|x_i -x_j\|_2), \qquad V(r) = e^{-2a(r-1)}-2e^{-a(r-1)},
\end{equation}
with stiffness parameter $a=4$.
\begin{figure}
\centering
  \includegraphics[width=\textwidth]{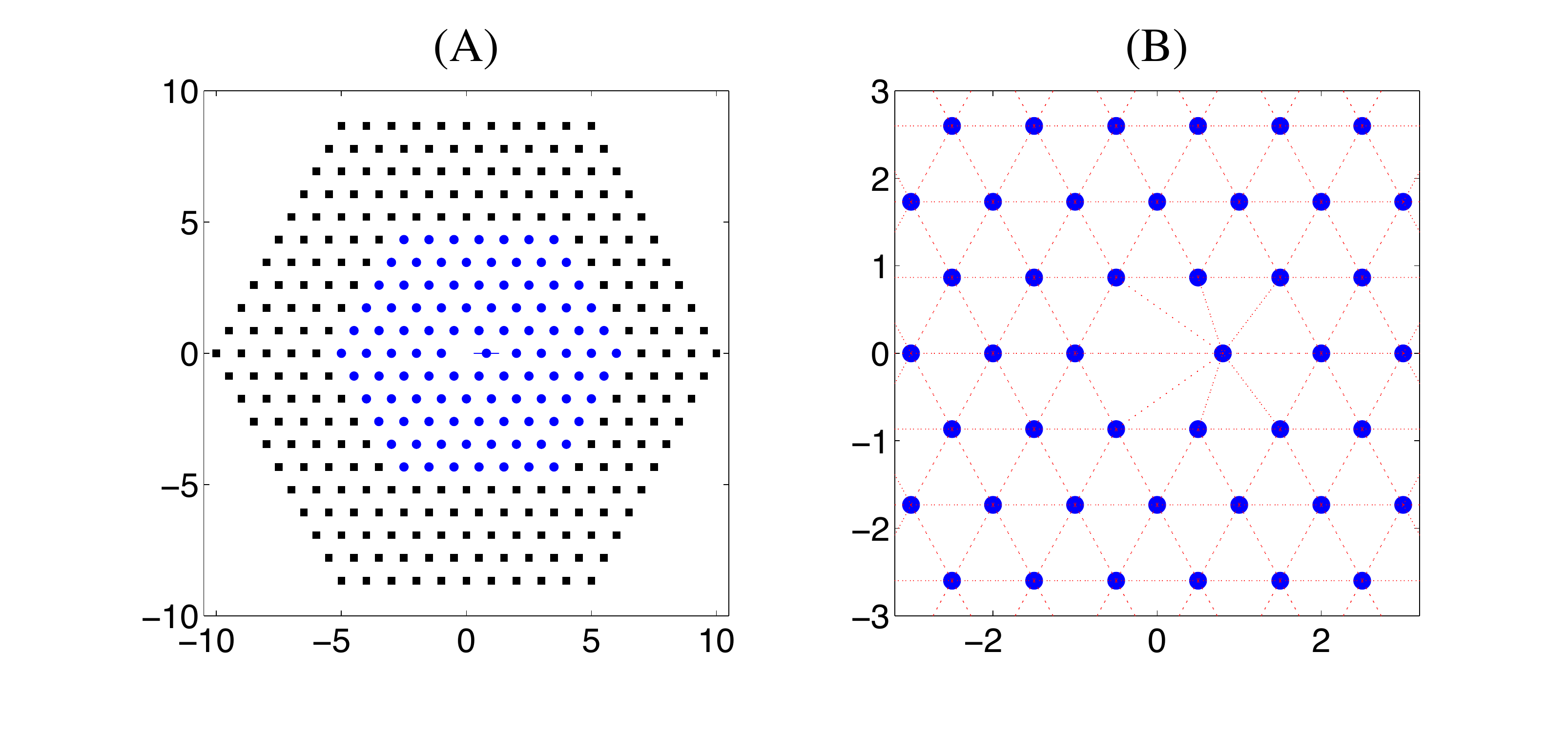}
  \caption{ \label{fig:Prob2Fig1} Initial configuration of the atoms in
    the vacancy diffusion problem (Test 2) . Black squares are fixed atoms while
    blue circles are atoms which move freely.\ (A) The initial dimer
    orientation is selected so that the translated atom has an
    orientation along the $y = 0$ direction, and is zero for all other
    atoms.\ (B) The Delaunay $\mathcal{T}_k$ triangulation used for the
    connectivity norm.}
\end{figure}

This test case demonstrates the importance of selecting the correct norm
for high-dimensional problems. The experiment is run both using the
generic $\ell_2$ norm (no preconditioner), as well as a `connectivity'
norm. Such a norm can be defined based on the Delaunay triangulation of
the atomistic positions (Figure \ref{fig:Prob2Fig1}(B))
\begin{displaymath}
  \langle M_k u, u \rangle = \int |\nabla I_{\mathcal{T}_k} u|^2,
\end{displaymath}
where $\mathcal{T}_k$ is the triangulation depicted in the figure and
$I_{\mathcal{T}_k}$ the associated nodal interpolant.
  
Figure \ref{fig:Prob2Fig2} demonstrates the convergence to the saddle
with different numbers of free atoms $nA$ (giving different
dimensionality of the system) in the two norms for the linesearch
dimer.
\begin{figure}
\centering
  \includegraphics[width=\textwidth]{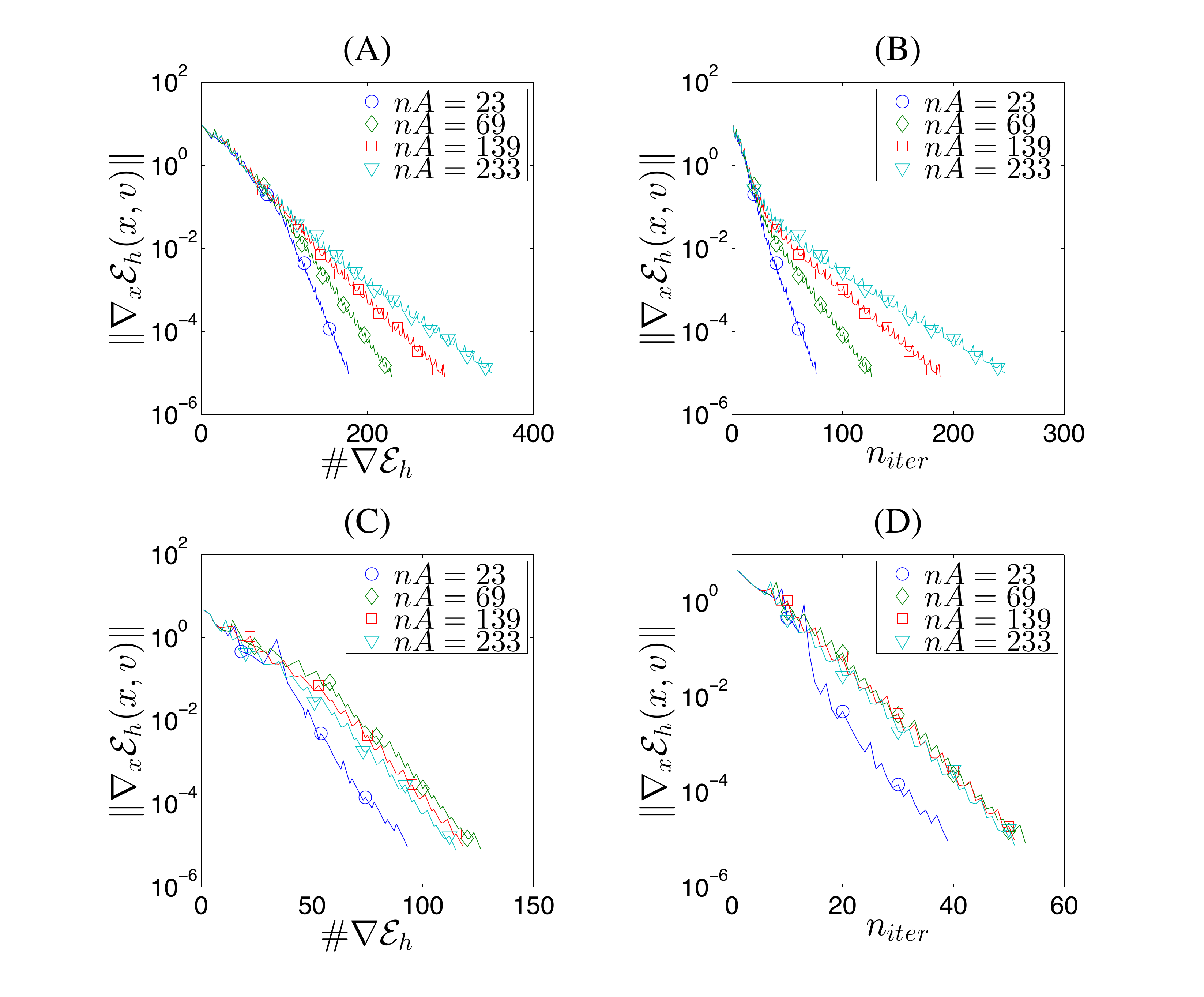}
  \caption{ \label{fig:Prob2Fig2} Convergence of the linesearch dimer to the
    saddle in the vacancy diffusion problem (Test 2) 
   with (A),(C) the $\ell_2$ norm
    and (B),(D) connectivity norm versus the number of force
    evaluations and dimer iterations for increasing numbers of free atoms.}
\end{figure}
We can also observe the benefit of the linesearch vs a simple
dimer scheme when using the connectivity norm (Figure
\ref{fig:Prob2Fig3}). The linesearch dimer selects very efficient
stepsizes with no a-priori information, while the simple dimer method
might exhibit either slow convergence, or no convergence, if the fixed
steps are poorly chosen.
\begin{figure}
\centering
  \includegraphics[width=\textwidth]{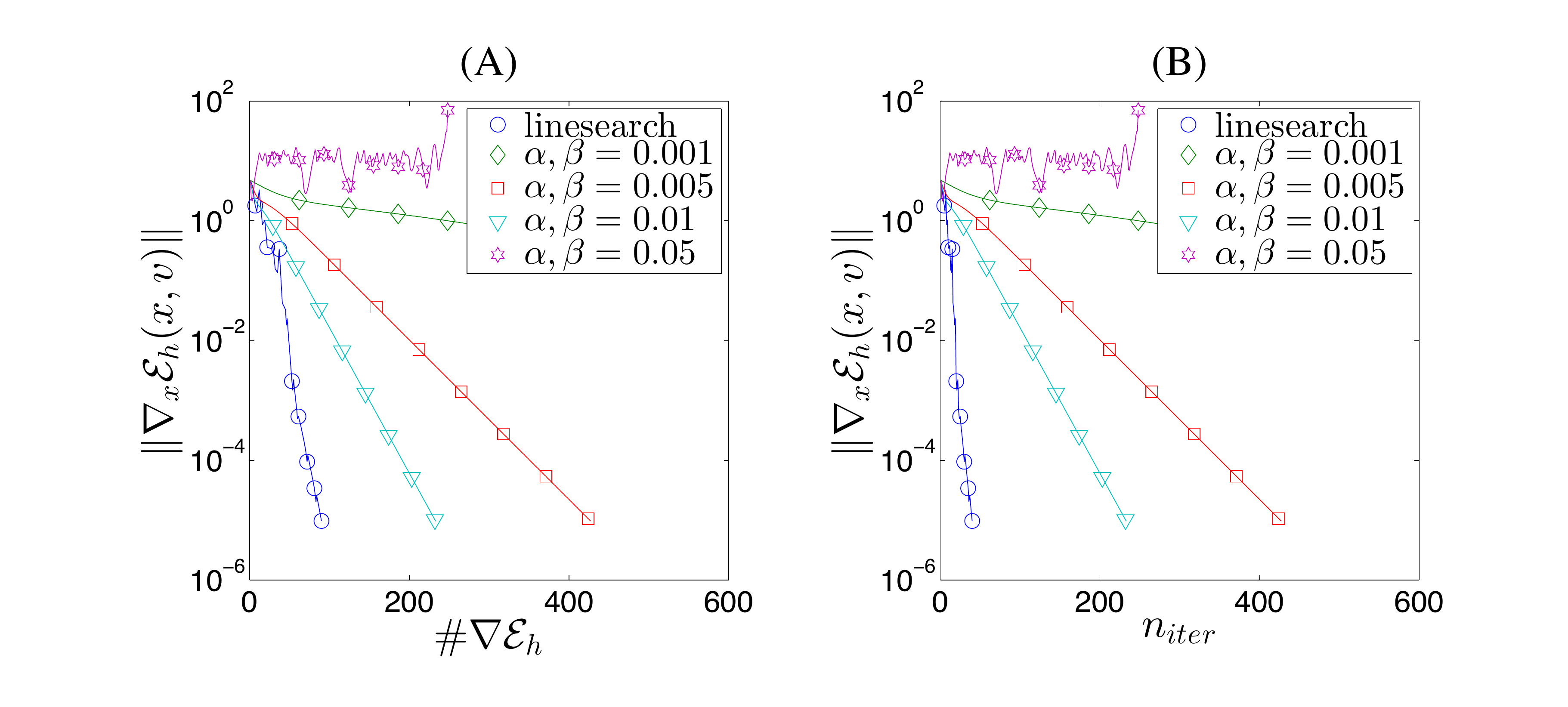}
  \caption{ \label{fig:Prob2Fig3} Convergence of the linesearch dimer vs
    the simple dimer method for Test 2 some choice of the of simple dimer step
    sizes with $nA=69$ using the connectivity norm.}
\end{figure}

\subsection{Test 3: A Phase Field Example}

Our final example is based on a simple phase field model where the
global energy is given by,
\begin{equation}
E(u) = \int_{\Omega} \frac{\epsilon}{2}\left|\nabla u \right|^2 + 
 \frac{1}{2\epsilon}(u^2-1)^2.
\end{equation}
In our test $\Omega$ is the unit square, and the boundary conditions are,
\begin{equation}
u(x) = \cases{
  -1, & x_1 \in \{0,1\} \\ 1, & x_2 \in
    \{0,1\}.
  }
\end{equation}

There are 2 minima of such an energy, these are given in Figure
\ref{fig:Prob3Fig1}(A),(B). 
The saddle between these two minima is given in Figure \ref{fig:Prob3Fig1}(C).

\begin{figure}
\centering
  \includegraphics[width=\textwidth]{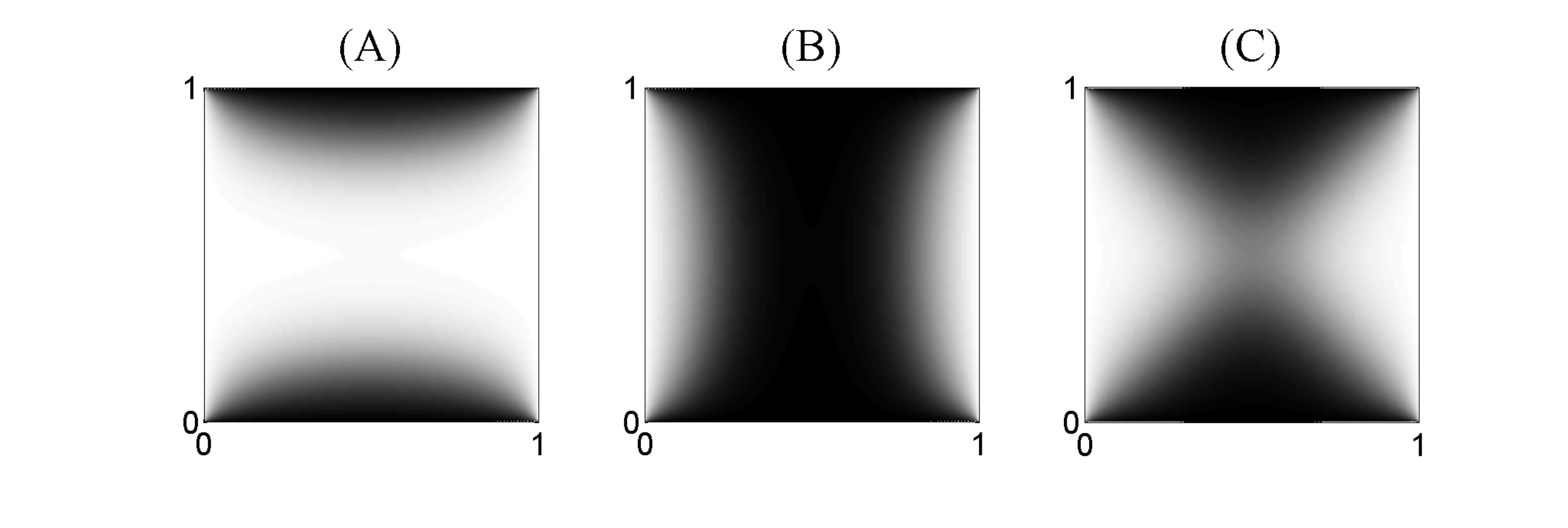}
  \caption{ \label{fig:Prob3Fig1}  Minima (A,B) and saddle point (C) 
  of the phase field problem (Test 3) with $\epsilon = 1/10$. The shading is linearly interpolated between {white(-1)} and {black(1)}. }
\end{figure}

A possible choice for a preconditioner for this system is a stabilized
Laplacian,
\begin{equation}
\label{stablap}
P = \epsilon \Delta + \frac{1}{\epsilon} I.
\end{equation}

In order to compute either a minimum or a saddle point for such a system we
triangulate the domain into a variable number of elements, thereby
creating a discrete system of variable dimensionality.
In our tests we take the initial dimer point as a small random
perturbation of one of the local minima, and the initial dimer
orientation is the metric inverted against a vector of ones.

In Figure \ref{fig:Prob3Fig2} we demonstrate the necessity of using a
preconditioner to solve this problem using the simple dimer
method. When using the preconditioner \eqref{stablap},
the algorithm performs well when the step size is chosen appropriately.
We observe the expected behaviour,
that there exists an optimal step size where convergence is fastest,
and beyond that step size the dimer diverges. In fact we observe that
the stabilized Laplacian metric is so effective, that the optimal step
size seems very close to the unit step. If the $\ell_2$ norm (identity
preconditioner) is used then for all step sizes tested the dimer
diverges, indicating that at best a very small step would need to be
chosen for convergence.

\begin{figure}
\centering
  \includegraphics[width=\textwidth]{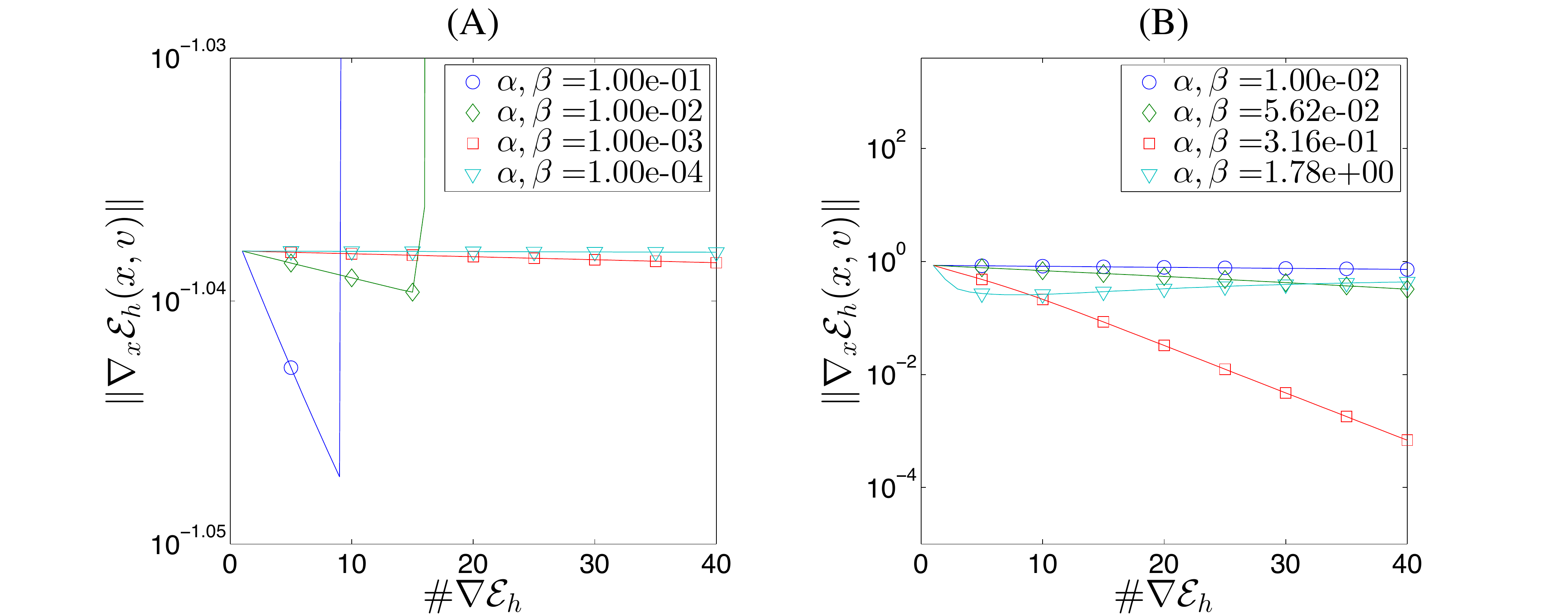}
  \caption{ \label{fig:Prob3Fig2} Convergence of the simple dimer to
    the saddle in the phase field problem (Test 3) with (A) the $\ell_2$ metric
    and (B) the stabilized Laplacian metric where $\epsilon=1/10$ for
    a triangulation with 3485 degrees of freedom. }
\end{figure}

In Figure \ref{fig:Prob3Fig3} we demonstrate that the used of the scaled
Laplacian metric for different system sizes. We observe that the use of
this metric gives almost perfect scale invariance.
\begin{figure}[h]
\centering
  \includegraphics[width=\textwidth]{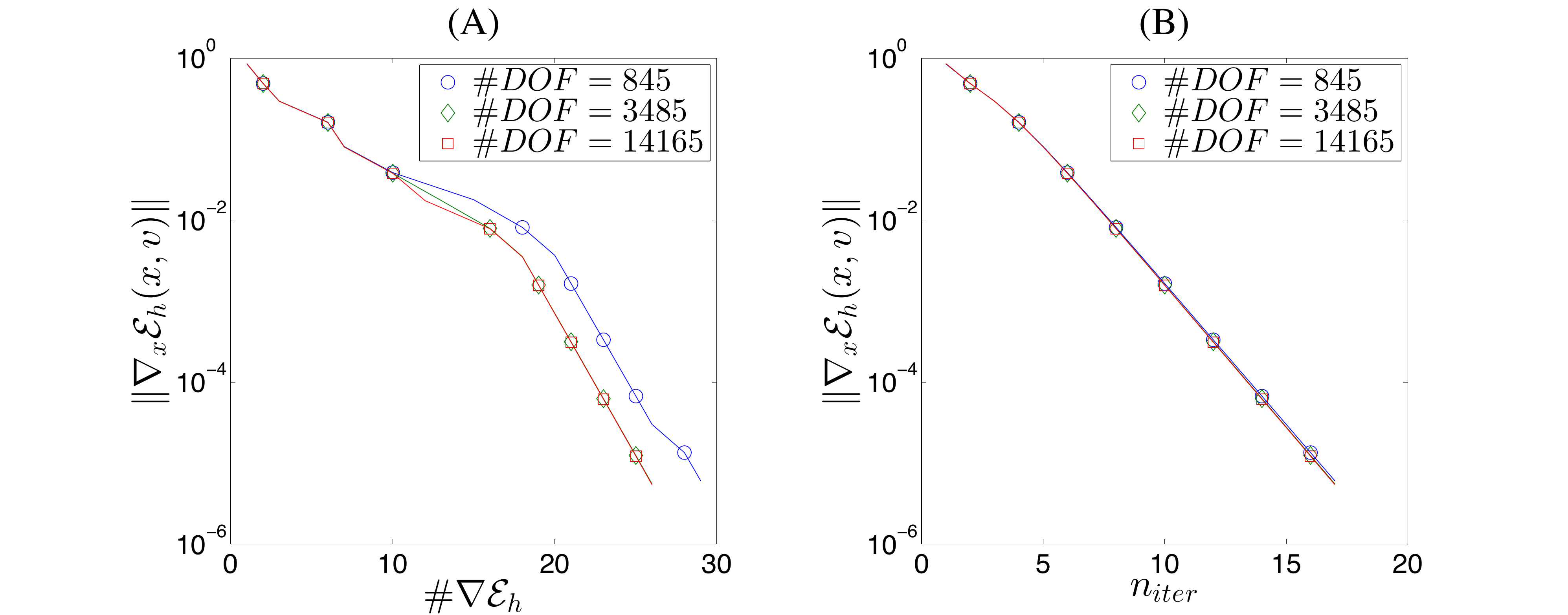}
  \caption{ \label{fig:Prob3Fig3} Convergence of the linesearch dimer
    to the saddle in the phase field problem (Test 3) with the stabilized
    Laplacian metric and triangulations of varying coarseness.}
\end{figure}

In Figure~\ref{fig:Prob3Fig4} we give the results of applying the simple
and linesearch dimers with varying $\epsilon$; the coarseness of the
discretization in each experiment is chosen such that $\Delta x \approx
\epsilon/5$. In some of these cases the linesearch dimer fails due to
rounding error. Specifically, due to rounding error in the naive
implementation of the energy function (simple summation over the
elements), the translation step fails to find a sufficient decrease in
the dimer energy, the step size selected shrinks to zero (to rounding
error) and the method stagnates. In order to correct this a more robust
method of evaluating the energy or a more advanced optimization
algorithm should be implemented which can either choose better
linesearch directions or more robustly deal with numerically zero energy
changes.

We also observe, in the case $\epsilon = 1/30$ that the rate of
convergence of even the simple dimer changes once the residual moves
below a certain value. We are unable to give a satisfactory
explanation for this effect, but speculate that the singularity in the
boundary condition (which excludes admissible $H^1$-states) might be
the case. (In particular, we observed that this behaviour is
independent of the mesh coarseness and of the dimer length.)
\begin{figure}
\centering
  \includegraphics[width=\textwidth]{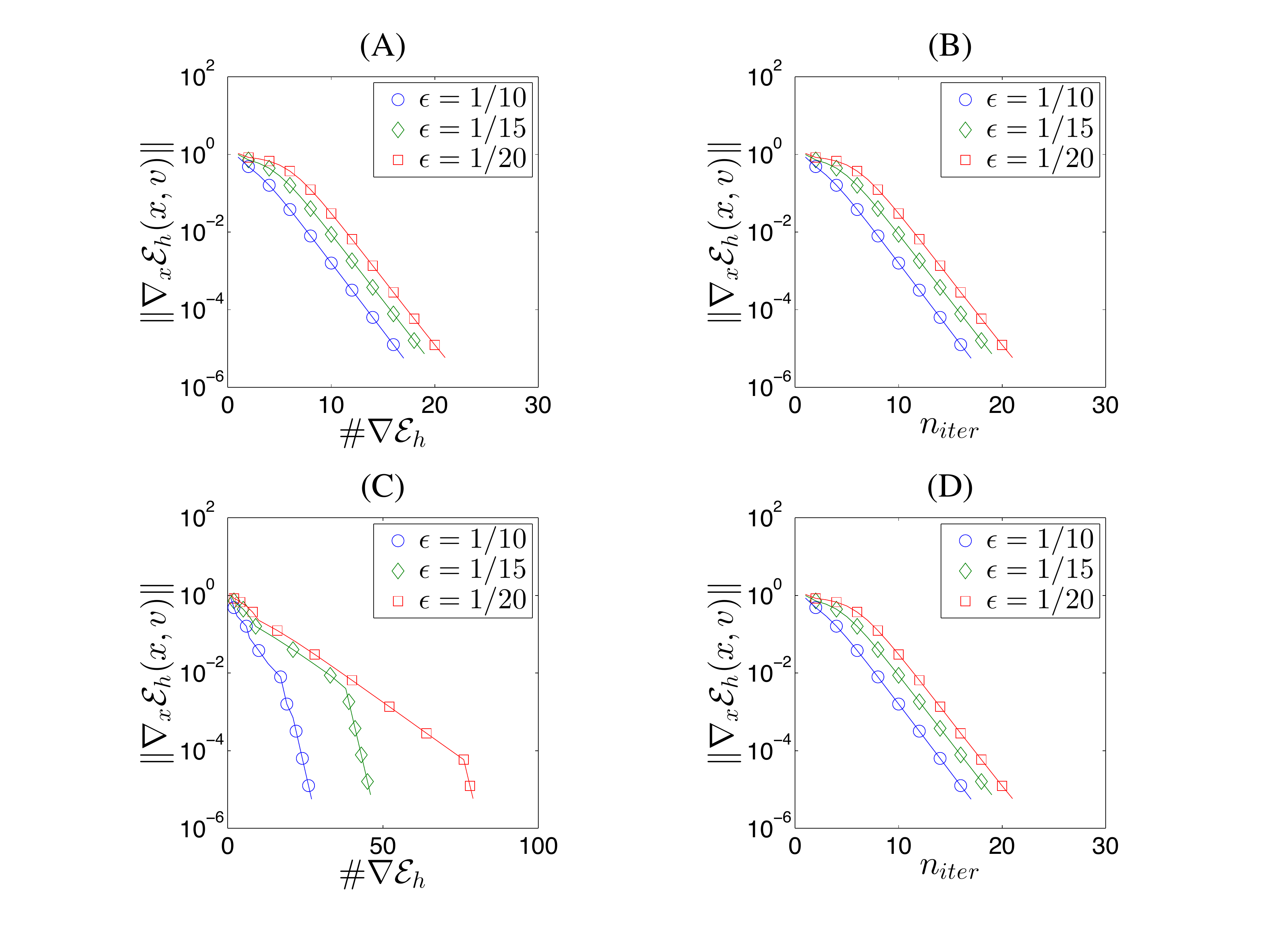}
  \caption{ \label{fig:Prob3Fig4}  Convergence  to the
    saddle in the phase field problem (Test 3) using the stabilized Laplacian 
    metric with (A),(B) the simple dimer with unit step length and (C),(D) the 
    linesearch dimer for a triangulation with 2405,9805,22205 degrees of 
    freedom for the respective choices of $\epsilon$.}
\end{figure}

\section{Conclusions}

We have described a dimer method for finding a saddle point in which
the dimer length $h$ is not required to shrink to zero, but which
converges to a point that lies within $O(h^2)$ of a saddle. We have
enhanced this algorithm with a lineasearch to improve its robustness,
and use the observation that the dimer method may be formulated and
applied in a general Hilbert space to allow preconditioning that
improves the method's efficiency. The linesearch uses a local merit
function.  Unfortunately our particular merit function may not lead to
global convergence of the iterates, and it is an open question as to
whether there is another merit function that ensures global
convergence. We have illustrated the positive effects of our
algorithms on three realistic examples.

\appendix

\section{Proofs}

\subsection{Proof of Proposition \ref{th:dimer_saddle}}
\label{sec:proof_dimer_saddle}
%
We prove the result using the inverse function theorem. We write
\eqref{eq:dimer_saddle_system} as $F(x_h, v_h, \lambda_h) = 0$ and
show that $\|F(x_*, v_*, \lambda_*) \| \leq C h^2$ and that $\D F
(x_*, v_*, \lambda_*)$ is an isomorphism with bounds independent of
$h$. The inverse function theorem then yields the stated result.

{\it Residual estimate. } Let the residual components be 
  \begin{align*}
    r_x &:= F_x(x_*, v_*, \lambda_*) = \smfrac12 \b(\nabla E(x_* + h
    v_*) + \nabla
    E(x_* - h v_*)\b), \\
    r_v &:= F_v(x_*, v_*, \lambda_*) = \smfrac{1}{2h} (\nabla E(x_*+h
    v_*) - \nabla E(x_* - h v_*) )
    - \lambda_* v_*, \\
    r_\lambda &:= F_\lambda(x_*, v_*, \lambda_*) = \smfrac12 (\| v_* \|^2 - 1).
\end{align*}
Then,
\begin{align*}
  r_x &= \nabla E(x_*) + \smfrac12 \nabla^2 E(x_*) (h v_* - h v_*) +
  O(h^2) =
  O(h^2), \\
  r_v &= \nabla^2 E(x_*) v_* - \lambda_* v_* + \smfrac1h \b( \nabla^3
  E(x_*) [h v_* \otimes h v_* - h v_*
  \otimes h v_*] + O(h^2) = O(h^2), \\
  r_\lambda &= 0.
\end{align*}
Thus, $\| F(x_*, v_*, \lambda_*) \| \leq C h^2$.

{\it Stability. } $\D F(x_*, v_*, \lambda_*)$ can be written in the
form
\begin{align*}
  \D F(x_*, v_*, \lambda_*)&= 
  \left[ \begin{matrix} 
      \frac{\D^2 E(x_*+hv_*) + \D^2 E(x_* -hv_*)}{2} & 
      h \frac{\D^2 E(x_* +hv_*)- \D^2
      E(x_* -h v_*)}{2} & 0 \\
       \frac{\D^2 E(x_* +hv_*)- \D^2 E(x_* -h v_*)}{2h} &  
       \frac{\D^2 E(x_* +h v_*) + \D^2 E(x_* -h v_*)}{2}  -
     \lambda_* I & v_* \\
      0 & v_*^T & 0
    \end{matrix} \right] \\
  &=   \left[ \begin{matrix}
      \D^2 E(x_*) & 0 & 0 \\
      \D^3 E(x_*) \cdot v_* &  \D^2 E(x_*) - \lambda_* I & v_* \\
      0 & v_*^T & 0
    \end{matrix} \right] + O(h^2) =: {\bf A} + O(h^2).
\end{align*}
where we used \eqref{eq:err_Dx2Eh}, \eqref{eq:defn:errorH} and
\eqref{eq:err_DxDvEh}.
By assumption, $\D^2 E(x_*)$ is an isomorphism on $X$. Since, also by
assumption, $\lambda_*$ is a simple eigenvalue, the block
\begin{equation}
  \label{eq:prf_dimersaddle:eigenblock}
  \left[ \begin{matrix}
      \D^2 E(x_*) -
      \lambda_* I & v_* \\
      v_*^T & 0
    \end{matrix} \right]  
\end{equation}
is an isomorphism on $X \times \R$ as well. Thus, ${\bf A}$ is an
isomorphism on $X \times X \times \R$ and consequently, for all $h$
sufficiently small, $\D F(x_*, v_*, \lambda_*) = A + O(h^2)$ is also
an isomorphism, with a uniform bound on its inverse.

Thus, the inverse function theorem shows that there exist a radius
$r_0 > 0$ and a dimer length $h_0 > 0$, such that, for $h \leq h_0$,
there exists a unique solution $(x_h, v_h, \lambda_h)$ to
\eqref{eq:dimer_saddle_system} in a ball of radius $r_0$ about $(x_*,
v_*, \lambda_*)$, satisfying the estimate
\eqref{eq:dimer_saddle_error}.


\subsection{Proof of Theorem \ref{th:local_conv_dimer} (a)}
\label{sec:prfa}
Fix $r$ and $h_0$ sufficiently small so that Theorem \ref{th:dimer_saddle}
applies. Let $e_k := x_k-x_h, f_k := v_k-v_h$ and $r_k :=
\sqrt{\|e_k\|^2+\|f_k\|^2}$, so that trivially $\|e_k\| \leq r_k$ and 
$\|f_k\| \leq r_k$.

\begin{lemma}
  \label{th:prfa:lemma1}
  Let $p := -(I- 2 v_k\otimes v_k) \D_x\E(x_k, v_k)$ and $s :=
  -(I-v_k\otimes v_k) H_h(x_k; v_k)$, then, under the assumptions of
  Theorem \ref{th:local_conv_dimer},
  \begin{align}
    \label{eq:prfa:exp_p}
    p &= - A e_k + O(r_k^2 + h^2 r_k), \qquad \text{and}  \\
    \label{eq:prfa:exp_s}
    s &= - B e_k - C f_k + O(r_k^2 + h^2 r_k),
  \end{align}
  where the operators $A$ and $C$ are defined in \eqref{eq:prfa:defn_ABC}
 and $B$ is a bounded linear operator.
\end{lemma}
\begin{proof}
  To prove \eqref{eq:prfa:exp_p} we first note the following
  identities which are easy to establish:
  \begin{align}
    \D_x \E_h(x, v) &= \D E(x) + O(h^2), \notag\\
    \D_x \E_h(x_k, v_k) - \D_x \E_h(x_h, v_h) &= O(r_k), \notag\\
    v_k \otimes v_k - v_h \otimes v_h &= O(r_k), \notag\\
    \D_x^2 \E_h(x_h, v_h) &= \D^2 E(x_h) + O(h^2) = \D^2 E(x_*) +
    O(h^2) \notag \\ 
    \D_x \D_v \E_h(x_h, v_h) &= \smfrac{1}{2} \b( \D^2 E(x_h+hv_h) -
    \D^2 E(x_h-hv_h) \b) = O(h^2).     \label{eq:nablavxh}
  \end{align}
  Using these identities, we can expand
  \begin{align*}
    p &= - (I-2 v_k\otimes v_k) \b( \D_x \E_h(x_k, v_k) - \D_x \E_h(x_h, v_h)\b), \\
    &= -(I-2v_h \otimes v_h) \b(  \D_x^2 \E_h(x_h, v_h) e_k + \D_x\D_v
    \E_h(x_h, v_h) f_k \b) + O(r_k^2) \\
    &=  - (I-2 v_h\otimes v_h) \D^2 E(x_h) e_k + O(r_k^2 + h^2 r_k) \\
    &= - (I-2v_*\otimes v_*) \D^2 E(x_*) e_k + O(r_k^2 + h^2 r_k) \\
    &= - A e_k + O(r_k^2+h^2 r_k).
  \end{align*}

  To prove \eqref{eq:prfa:exp_s}, we first note that, with $\|v\|=1$,
  \begin{align*}
    H_h(x; v) &= \mint_{-1}^1 \D^2 E(x+thv) \dt \, v = \D^2 E(x) v +
    O(h^2), \\
    H_h(x_h; v_h) &= \D^2 E(x_h) v_h + O(h^2) = \D^2 E(x_*) v_* +
    O(h^2),  \\
    H_h(x_k; v_k) - H_h(x_h; v_h) &= \mint_{-1}^1 \B( \D^2 E(x_k+th
    v_k) - \D^2 E(x_h+th v_h) \B) \dt \, v_k  \\
    & \qquad \qquad + \mint_{-1}^1 \D^2
    E(x_h+thv_h) \dt (v_k - v_h) \\
    &= \mint_{-1}^1 \B( \D^3 E(x_h+thv_h) \b[ (x_k-x_h) + th (v_k-v_h)
    \b] \dt v_h \\
    & \qquad \qquad + \D^2 E(x_*) (v_k - v_h) + O( r_k^2 + h^2 r_k) \\
    &= (\D^3 E(x_*) v_*) e_k + \D^2 E(x_*) f_k + O(h^2 r_k + r_k^2),
  \end{align*}
  where we interpret $\D^3 E(x) \cdot v \in L(X)$ via the action $w
  \cdot ((\D^3 E(x) \cdot v) z) = \lim_{t \to 0} t^{-1} w \cdot ((\D^2
  E(x+tv)-\D^2 E(x)) z)$. Finally, we also have
  \begin{align*}
    (v_k\otimes v_k - v_h\otimes v_h) H_h(x_h;v_h) &= (v_k\otimes v_k
    - v_h\otimes v_h) \D^2 E(x_*) v_* + O(h^2 r_k) \\
    &= \lambda_* (v_k\otimes v_k - v_h\otimes v_h) v_* + O(h^2 r_k) \\
    &= \lambda_* (v_k\otimes v_k - v_h\otimes v_h) v_h + O(h^2 r_k) \\
    &= \lambda_* (v_k-v_h) + \lambda_* v_k ((v_k-v_h) \cdot v_h) +
    O(h^2 r_k) \\
    &= \lambda_* f_k + O(r_k^2 + h^2 r_k).
  \end{align*}
  In the very last line we also used the fact that $v_k \cdot v_h - 1
  = \smfrac12 \|v_k - v_h\|^2$.

  Using these identities, we can compute
  \begin{align*}
    s &= - (I-v_k\otimes v_k) H_h(x_k; v_k) \\
    &= (I-v_h \otimes v_h) H_h(x_h; v_h) - (I-v_k\otimes v_k) H_h(x_k;
    v_k) \\
    &= -(I-v_k\otimes v_k) \b( H_h(x_k; v_k) - H_h(x_h; v_h) \b) + (v_k
    \otimes v_k - v_h \otimes v_h) H_h(x_h; v_h) \\
    &= - (I-v_k\otimes v_k) \b( (\D^3 E(x_*)v_*) e_k + \D^2 E(x_*) f_k
    \b)  + O(h^2 r_k + r_k^2) \\
    & \qquad \qquad + \lambda_* f_k + O(r_k^2+h^2 r_k) \\
    &=: -B e_k + \b[\lambda_* I - (I-v_* \otimes v_*) \D^2 E(x_*)\b] f_k +
    O(r_k^2+h^2 r_k)  \\
    &= - B e_k - C f_k + O(r_k^2+h^2 r_k).  \qedhere
  \end{align*}
\end{proof}


From Lemma \ref{th:prfa:lemma1} it follows in particular that $s =
O(r_k)$. Hence, Taylor expansions of sine and cosine in the identity
\begin{displaymath}
  v_{k+1} = \cos\b( \|s\| \beta_k) v_k + \sin(\|s\| \beta_k) \smfrac{s}{\|s\|},
\end{displaymath}
yield
\begin{displaymath}
  f_{k+1} = f_k + \beta_k s + O(\beta_k^2 s_k^2) 
\end{displaymath}
Using Lemma \ref{th:prfa:lemma1}, the identity $e_{k+1} = e_k + \alpha_k p$,
and the fact that $\beta_k$ is bounded,
we therefore obtain identity \eqref{eq:prfa:mainstep} in the proof outline.

Upon defining 
\begin{displaymath}
  {\bf A}_k := \left( \begin{matrix} \alpha_k A & 0 \\ \beta_k B &
      \beta_k C  \end{matrix} \right) \quad \text{and} \quad
  {\bf e}_{k} =
  \left( \begin{matrix} e_k \\ f_k \end{matrix} \right)
\end{displaymath}
\eqref{eq:prfa:mainstep} reads
\begin{equation}
  \label{eq:prfa_mainstep-v2}
  {\bf e}_{k+1} = (I-{\bf A}_k) {\bf e}_k + {\bf t}_k,
\end{equation}
where
\begin{displaymath}
  \| {\bf t}_k \| \leq C_{\rm t} (\alpha_k+\beta_k) (h^2 + r_k) r_k.
\end{displaymath}

Due to the fact that $A$ is symmetric and positive definite, it
follows that, for $\bar\alpha,\bar\beta$ chosen sufficiently small and
$\underline{\alpha} = \inf_k \alpha_k, \underline{\beta} = \inf_k
\beta_k > 0$, the spectrum of $I - {\bf A}_k$ is real and belongs to
$[0, 1 - \epsilon]$ for some $\epsilon > 0$, that depends on
$\underline{\alpha}, \underline{\beta}$. This will be crucial later in
the proof.

  \begin{lemma}
    Let ${\bf P}_{n,k} := \prod_{i = n}^k (I-{\bf A}_i)$, for $0 \leq
    n \leq k$, then there exist constants $C_1 > 1, \mu \in (0, 1)$
    such that
    \begin{equation}
      \label{eq:est_prod_Ak}
      \| {\bf P}_{n,k} \| \leq C_1 \mu^{k-n+1}.
    \end{equation}
  \end{lemma}
  \begin{proof}
    We assume, without loss of generality, that $n = 0$. First, we
    note that the diagonal blocks
    \begin{displaymath}
      A_k := [{\bf P}_{0,k}]_{xx} = \prod_{i=0}^k (I-\alpha_k A),
      \quad \text{and} \quad 
      C_k := [{\bf P}_{0,k}]_{vv} = \prod_{i=0}^k (I-\beta_k C),
    \end{displaymath}
    and it is easy to see that, for 
    $\underline{\alpha} \leq \alpha_i \leq \bar\alpha, 
     \underline{\beta} \leq \beta_i \leq \bar\beta$ chosen sufficiently small,
    that
    \begin{equation}
      \label{eq:ac-bnd}
      \| A_k \| \leq \bar\mu^k \qquad \text{and} \qquad 
      \| C_k \| \leq \bar\mu^k.
    \end{equation}
    where $\bar\mu := \max \b( 1 - \underline{\alpha} \inf \sigma(A),
                           1 - \underline{\beta} \inf \sigma(C)\b) \in (0,1)$.

    Since the off-diagonal block $[{\bf P}_{0,k}]_{xv} = 0$,
    it remains to estimate the off-diagonal block $B_k := [{\bf
      P}_{0,k}]_{vx}$. We use induction over $k$. Let 
      $C_* := \bar\mu^{-1} \underline\beta \| B\|$ and suppose that 
    \begin{equation}
      \label{eq:b-bnd}
    \|B_k \| \leq C_* k \bar\mu^k.
    \end{equation}
    Then, using
    \begin{displaymath}
      B_{k+1} = -\beta_{k+1} B A_k + (I-\beta_{k+1} C) B_k,
    \end{displaymath}
    as well as $\| I-\beta_{k+1} C \| \leq \bar\mu$ we can estimate
    \begin{displaymath}
      \|B_{k+1} \| \leq  \underline\beta \| B\| \bar\mu^k + 
                          \bar\mu C_* k \bar\mu^k 
       = C_* (k+1) \bar\mu^{k+1},
    \end{displaymath}
    which establishes the induction since the result is true by
    definition of $C_*$ when $k = 0$.  

    Now pick $\tau > 1$ so that $\mu := \tau \bar\mu < 1$.
    Then \eqref{eq:ac-bnd} give that 
    $\| A_k \| \leq \mu^k$ and $\| C_k \| \leq \mu^k$, while
    it follows from \eqref{eq:b-bnd} and by maximizing $x \tau^{-x}$ that
    \begin{displaymath}
    \|B_k \| \leq C_* ( k / \tau^k ) \mu^k \leq 
    ( C_* / \tau \ln \tau ) \mu^k.
    \end{displaymath}
    The result now follows from the inequality
    \begin{displaymath}
      \| {\bf P}_{0,k} \| \leq \b( \|[{\bf P}_{0,k}]_{xx} \|^2 +
      \|[{\bf P}_{0,k}]_{vx} \|^2 +  
      \|[{\bf P}_{0,k}]_{vv} \|^2 \b)^{1/2},
    \end{displaymath} 
  and by defining $C_1 := \mu^{-1}  \sqrt{2 + ( C_* / \tau \ln \tau )^2}$.

  \end{proof}

  It is straightforward to prove that
  \begin{displaymath}
    {\bf e}_{k+1} = {\bf P}_{0,k} {\bf e}_0 + {\bf t}_{k} + {\bf
      P}_{k,k} {\bf t}_{k-1} + {\bf P}_{k-1,k} {\bf t}_{k-2} + \dots +
    {\bf P}_{1,k} {\bf t}_0, 
  \end{displaymath}
  which implies
  \begin{displaymath}
    \| {\bf e}_{k+1} \| \leq C_1 \mu^{k+1} \| {\bf e}_0 \| +
    \sum_{i = 0}^{k} C_1 \mu^i \| {\bf t}_{k-i} \|,
  \end{displaymath}
  that is,
  \begin{equation}
    \label{eq:prfa_almostdone}
    r_{k+1} \leq C_2\bg( \mu^{k+1} r_0 + \sum_{i = 0}^k
    \mu^{i+1} (r_{k-i}+h^2) r_{k-i} \bg),
  \end{equation}
  for some $C_2 \geq C_1$.

  We make another induction hypothesis that, 
  \begin{equation}
    \label{eq:prfa_laststep}
    r_i \leq C_3 \gamma^i r_0,
  \end{equation}
  where $\gamma \in (\mu, 1)$ and $C_3 > C_2$ are arbitrary. 
  The statement \eqref{eq:prfa_laststep} is clearly true for $i =
  0$. Assume now that it holds for $i = 0, \dots, k$, then
  \eqref{eq:prfa_almostdone}, and using $\mu/\gamma < 1$ yields
  \begin{align*}
    r_{k+1} &\leq C_2 \gamma^{k+1} r_0 \bg(
    \b(\smfrac{\mu}{\gamma}\b)^{k+1} 
    + \sum_{i = 0}^k \b(\smfrac{\mu}{\gamma}\b)^{i+1} \B( C_3
    \gamma^{k-i} r_0 + h^2 \B) C_3 \bg) \\
    &\leq C_2\gamma^{k+1}r_0 \bg( 1 + \frac{C_3^2 r_0 + C_3
      h^2}{1-\mu/\gamma} \bg).
  \end{align*}
  Since $C_3 > C_2$, upon choosing $r_0, h$ sufficiently small, we can
  achieve that
  \begin{displaymath}
    C_2  \bg( 1 + \frac{C_3^2 r_0 + C_3
      h^2}{1-\mu/\gamma} \bg) \leq C_3,
  \end{displaymath}
  hence \eqref{eq:prfa_laststep} holds also for $i = k+1$. This
  completes the proof of \eqref{eq:prfa_laststep} and hence of 
  Theorem~\ref{th:local_conv_dimer}~(a).

\subsection{Proof of Theorem \ref{th:local_conv_dimer} (b)}
\label{sec:prfb}
We begin with a basic auxiliary result.

\begin{lemma}
  \label{th:aux_lemma_Eh}
  Let $(x_*, v_*, \lambda_*)$ be an index-1 saddle and $\mu_* :=
  \inf_{\|w\| = 1, w \perp v_*} (\D^2 E(x_*) w) \cdot w > 0$. Then,
  there exists $r > 0$ and $h_0 > 0$ (chosen independently of one
  another) such that the following hold:
 \begin{itemize}
 \item[(i)] If $x \in B_r(x_*)$ then $\D^2 E(x)$ has index-1 saddle
   structure and, if $(\lambda, v)$ is the smallest eigenpair of $\D^2
   E(x)$, then $\lambda \leq \lambda_*/2$ and $(\D^2 E(x) w) \cdot w
   \geq \mu_*/2 \|w\|^2$ for $w \perp v$.

 \item[(ii)] $V(x)$ is well-defined for all $x \in B_r(x_*)$ and
     $h \in (0,h_0]$, and $x \mapsto V(x) \in C^1(B_r(x_*))$.

 \item[(iii)] $E_h \in C^4(B_r(x_*))$ with
   \begin{align*}
     \D E_h(x) = \D_x \E_h(x, V(x)) & = \smfrac12 \b(\D E(x+h V(x)) +
     \D E(x-h V(x)) \b), \quad \text{and} \\
     \| \D^2 E_h(x) - \D^2 E(x) \| & \leq C_0 h^2
   \end{align*}
   for $x \in B_r(x_*)$, where $C_0$ is independent of $x, h$.

 \item[(iv)] Let $x \in B_r(x_*)$ and let $(\lambda, v)$ be the
   minimal eigenpair of $\D^2 E_h(x)$, then $\|v - V(x)\| \leq C h^2$,
   where $C$ is independent of $x, h$.  
%
 \end{itemize}
\end{lemma}
\begin{proof}
  For $r$ sufficiently small, the statement (i) is an obvious
  consequence of $x_*$ being an index-1 saddle and $\D^2 E$ locally
  Lipschitz continuous (which follows since $E \in C^4(X)$).

  The statement (ii) is proven similarly as Proposition
  \ref{th:dimer_saddle}, provided $h_0$ is chosen sufficiently small
  (depending on $\lambda_*, \mu_*$ and on derivatives of $E$ in
  $B_{2r}(x_*)$). The $C^1$-dependence of $V(x)$ on $x$ is a
    consequence of the implicit function theorem.

  The statement (iii) follows from an elementary Taylor expansion.

  Finally, (iv) follows again from (iii) and an argument analogous to
  Proposition \ref{th:dimer_saddle}.
\end{proof}

To complete the proof of Theorem \ref{th:local_conv_dimer}(b) we first
note that, according to Lemma \ref{th:aux_lemma_Eh}(ii), Step (2) of
Algorithm 2 is indeed well-defined, provided that we can ensure that
the iterates never leave a neighbourhood of $x_h$ and hence of
$x_*$. This will be established.

Fix $r, h_0$ sufficiently small so that Theorem \ref{th:dimer_saddle}
and Lemma \ref{th:aux_lemma_Eh} apply. Let $e_k := x_k - x_h$ and $r_k
:= \|e_k\|$. Let $s := - (I-2v_k\otimes v_k) \D E_h(x_k)$ be the
search direction and $\alpha_k > 0$ the step size, then 
\begin{displaymath}
  e_{k+1} = e_k + \alpha_k s
\end{displaymath}
Applying Lemma \ref{th:aux_lemma_Eh}(iii) we can expand
\begin{align*}
  \D E_h(x_k) &= \D E_h(x_k) - \D E_h(x_h) 
  = \D^2 \E_h(x_h) e_k + O(r_k^2) = \D^2 E(x_*) e_k + O( h^2 r_k + r_k^2).
\end{align*}
Arguing similarly as in the proof of part (a), 
\begin{align*}
  e_{k+1} &= e_k - \alpha_k (I - 2 v_k \otimes v_k) \D^2 E(x_*) e_k +
  O(h^2 r_k) \\
  &= e_k - \alpha_k (I - 2 v_* \otimes v_*) \D^2 E(x_*) e_k + O\b( h^2 r_k
  + r_k^2 \b) \\
  &= (I - \alpha_k A) e_k + O\b( h^2 r_k + r_k^2\b).
\end{align*}
For $\bar\alpha$ sufficiently small it is straightforward to see that
$\|I - \alpha_k A \| \leq 1 - \alpha_k \epsilon \leq 1 -
\underline\alpha \epsilon =: \gamma$, where $\epsilon > 0$ and $\gamma
\in (0, 1)$, and we therefore obtain
\begin{displaymath}
  r_{k+1} \leq (\gamma + C_1 h^2 + C_2 r_k) r_k.
\end{displaymath}
Clearly, for $h_0$ and $r_0$ chosen sufficiently small we obtain a
contraction, that is, $r_{k+1} \leq \gamma' r_k$ for some $\gamma' \in
(\gamma, 1)$.

This completes the proof of Theorem \ref{th:local_conv_dimer}(b).

\subsection{Contraction of steepest descent with linesearch}
\label{sec:analysis_sd}
In the section following this one, 
we will use statements about the steepest descent
method with backtracking that we suspect must be well known. Since 
we have been unable to find precisely the versions we require, we 
give both below, the latter with a full proof.

\begin{lemma}
  \label{th:contrsd_eucl}
  Let $X$ be a Hilbert space, $F \in C^3(X)$, and $x_* \in X$ with $\D
  F(x_*) = 0$ and $\D^2 F(x_*)$ positive definite, i.e., $u \cdot
  (\D^2 F(x_*) u) \geq \mu \|u\|^2$ for $\mu > 0$. Let $\| u \|_*^2 :=
  u \cdot (\D^2 F(x_*) u)$. Further, let $\bar\alpha >
  \underline{\alpha} > 0$, $\Theta \in (0, 1)$.

  Then, there exists $r > 0$ and $\gamma \in (0, 1)$, depending only
  on $\underline{\alpha}, \bar\alpha, \mu, \|\D^j F(x) \|$ for $x \in
  B_{1}(x_*)$, such that, for all $\alpha \in [\underline{\alpha},
  \bar\alpha]$ and for all $x \in B_r(x_*)$ satisfying the Armijo
  condition
  \begin{displaymath}
    F(x-\alpha \D F(x)) \leq F(x) - \Theta \alpha \| \D F(x) \|^2,
  \end{displaymath}
  we have
  \begin{displaymath}
    \b\| [x - \alpha \D F(x)] - x_* \b\|_* \leq \gamma \| x - x_* \|_*.
  \end{displaymath}
\end{lemma}
\begin{proof}
  The proof is a simplified version of the proof of Lemma
  \ref{th:contrSX} below.
\end{proof}

We now generalize the foregoing result to steepest descent on the unit
sphere. Convergence results for many methods on manifolds are given by
\cite[Chap.4]{AbsiMahoSepu08}. See specifically 
\cite[Thm.4.5.6]{AbsiMahoSepu08} and \cite{AbsiMahoTrum13}.

\begin{lemma}
  \label{th:contrSX}
  Let $X$ be a Hilbert space, $S_X := \{ u \in X \sep \|u\|=1\}$, $P_v
  := v \otimes v$ and $P_v' := I-P_v$ for $v \in S_X$. Let $F \in
  C^3(X)$,
  \begin{displaymath}
    g(v) := P_v' \D F(v) \qquad \text{and} \qquad
    H(v) := P_v' \D^2 F(v) P_v' - \b(\D F(v) \cdot v\b) I.
  \end{displaymath}
  We assume that there exists $v_* \in S_X$ and $\mu > 0$ such that
  \begin{align}
    g(v_*) = 0 
    \qquad \text{and} \qquad
    u\cdot \b( H(v_*) u \b) \geq \mu \|u\|^2 \quad \forall u \in X.
   \label{eq:mual}
  \end{align}
  Let $\|u\|_* := \sqrt{ u\cdot (H(v_*) u)}$.

  Let $\bar\alpha > 0$, $\Theta \in (0, 1)$, and for $v \in S_X$ and
  $\alpha \in \R$, denote
  \begin{displaymath}
    v_\alpha := \cos\b( \alpha \|g(v)\| \b) v - \sin\b(\alpha
    \|g(v)\| \b) \frac{g(v)}{\|g(v)\|}.
  \end{displaymath}
  Then, there exists $r > 0$ such that, for all $v \in B_r(v_*) \cap
  S_X$ and $\alpha \in (0, \bar\alpha]$ satisfying the Armijo
  condition
  \begin{displaymath}
    F(v_\alpha) \leq F(v) - \Theta \alpha \| g(v) \|^2,
  \end{displaymath}
  there exists a constant $\gamma(\alpha) \in [0, 1)$ such that
  \begin{displaymath}
    \b\| v_\alpha - v_* \b\|_* \leq \gamma(\alpha) \| v - v_* \|_*.
  \end{displaymath}

  The contraction factor $\gamma(\alpha)$ depends on $\alpha, \mu$ and
  on $\|\D^j F(x)\|, x \in B_1(v_*)$. Moreover, for any
  $\underline{\alpha} \in (0, \bar\alpha]$, $\sup_{\alpha \in
    [\underline\alpha,\bar\alpha]} \gamma(\alpha) < 1$.
\end{lemma}
\begin{proof}
  We first note that $\|\cdot\|_*$ is an equivalent norm, that is,
  there exists a constant $C_* = \|H(v_*)\|$ such that
  \begin{equation}
    \label{contrsd_sph:norm_equiv}
    \sqrt{\mu} \| u - u' \| \leq \| u - u' \|_* \leq C_* \|u - u'\| \qquad
    \forall u, u' \in X.
  \end{equation}
  
  {\it Step 1: Expansions. } There exists a constant $C_L$ such that,
  for all $v, w \in S_X$,
  \begin{align}
    \label{eq:contrsd_sph:lip_grad}
    \b\| g(v) - g(w) \b\| &\leq C_L \| v - w \|, \quad \text{and} \\
    \label{eq:contrsd_sph:lip_hess}
    \b\| \D^2 F(v) - \D^2 F(w) \b\| &\leq C_L \|v-w\|.
  \end{align}
  since $F \in  C^3(X)$ and $S_X$ is bounded.
  For $v \in S_X$ the identity
  \begin{equation}
    \label{eq:contrsd_sph:id1}
    v_* \cdot (v - v_*) = -\smfrac12 \| v -
    v_*\|^2    
  \end{equation}
  and $g(v_*) = 0$ yields
  \begin{equation}
    \label{eq:fddotvmvs}
    \D F(v_*) \cdot (v - v_*) = \D F(v_*) \cdot \b( (v_* \otimes v_*)
    (v-v_*) \b) = \b(- \smfrac12 \D F(v_*) \cdot v_* \b) \| v - v_* \|^2,
  \end{equation}
  and therefore,
  \begin{align}
    \notag
    F(v) &- F(v_*) = \D F(v_*) \cdot (v - v_*) 
    + \bg(\int_0^1 (1-t) \D^2
    F((1-t) v_* + t v) \dt (v - v_*)\bg) \cdot (v-v_*) \\
    \label{eq:contrsd_sph:expansion_1}
    &= \smfrac12 (v-v_*) \cdot \b(\bar{H}_v (v - v_*) \b) + 
    \smfrac12 (v-v_*) \cdot 
    \b( \b[ \D^2 F(v_*) - \b(\D F(v_*) \cdot v_*\b) I \b ] \b (v - v_*) \b),
  \end{align}
  where 
  \begin{displaymath} 
    \bar{H}_v := 2 \int_0^1 (1-t) [ \D^2 F((1-t) v_* + t v) - \D^2 F(v_*) ]\dt .
  \end{displaymath}
  But
  \begin{align*}
    (v-v_*) \cdot \B(\D^2 F(v_*) (v-v_*) \B) = (v-v_*) \cdot \B(
    P_{v_*}' \D^2 F(v_*) P_{v_*}' (v-v_*) \B) + O(\|v-v_*\|^3).
  \end{align*}
  since $(v_* \otimes v_*) (v-v_*) = O(\|v-v_*\|^2)$, and 
  thus we obtain from \eqref{eq:contrsd_sph:lip_hess} 
  and \eqref{eq:contrsd_sph:expansion_1}  that
  \begin{equation}
    \label{eq:contrsd_sph:energy_error}
    \smfrac12 \| v - v_*\|_*^2 - C_1 \| v - v_* \|^3 \leq F(v) -
    F(v_*) \leq \smfrac12 \| v - v_*\|_*^2 + C_1 \| v - v_*\|^3,
  \end{equation}
  for some constant $C_1$ that depends on $C_L$.  

  {\it Step 2: Bound on descent step. }
  The Lipschitz bound
  \eqref{eq:contrsd_sph:lip_grad} implies that, for all $v \in S_X$,
  \begin{align*}
    \| v_\alpha - v_* \| &\leq \|v - v_*\| + \b|1-\cos(\alpha \|g(v)\|)\b|
    + \b|\sin(\alpha \|g(v)\|)\b| \\
    &\leq \|v - v_*\| + \smfrac12 (\alpha \|g(v)\|)^2 + \alpha \|g(v)\| \\
    &\leq \|v - v_*\| + \smfrac12 \alpha^2 C_L^2 \|v - v_*\|^2 + 
    \alpha C_L \| v - v_*\|
    \\ &\leq 
    \b( 1 + \alpha^2 C_L^2 + \alpha C_L\b) \|v - v_*\| \\
    &=: c_3(\alpha) \| v - v_* \|,
  \end{align*}
  as $|1-\cos \theta| \leq \smfrac12 \theta^2$ and 
  $|\sin\theta| \leq \theta$ for $\theta \geq 0$,
  and $\| v - v_* \| \leq 2$.
  In particular, for $r > 0$
  \begin{equation}
    \label{eq:contrsd_sph:bound_step}
    \|v_\alpha - v_* \| \leq c_3(\bar\alpha) r \qquad \forall v \in B_r(v_*) \cap S_X.
  \end{equation}

  {\it Step 3. Bound on gradient. } To obtain an error estimate from
  the Armijo condition, we must bound $\|g(v)\|^2$ below. We write
  $v_t := (1-t) v_* + t v$, then 
  \begin{align}
    \notag
    g(v) &= g(v) - g(v_*) \\
    &= \int_0^1 \frac{\dd}{\dt} \B( (I - v_t \otimes v_t) \D F(v_t)
    \B) \dt \\
    \notag
    &= \int_0^1 \B( (I-v_t\otimes v_t) \D^2 F(v_t) (v-v_*) -
    \b((v-v_*)\otimes v_t + v_t \otimes (v-v_*)\b) \D F(v_t) \B) \dt
    \\
    &= (I-v_*\otimes v_*) \D^2 F(v_*) (v-v_*) \\
    \notag
    & \qquad - \b( (v-v_*) \otimes v_* + v_* \otimes (v-v_*) \b) \D
    F(v_*) + O(\|v-v_*\|^2) \\
    \notag
    &= H(v_*) (v-v_*)  +  (I-v_*\otimes v_*) \D^2 F(v_*) (v_*\otimes v_*)
    (v-v_*)  \\
    \notag
    & \qquad - \b( \D F(v_*) \cdot (v-v_*) \b) v_*  +
    O(\|v-v_*\|^2) \\
    \label{eq:contrsd_sph:g_exp}
    &= H(v_*) (v-v_*) + O(\|v-v_*\|^2),
  \end{align}
  where we used \eqref{eq:contrsd_sph:id1} 
  and \eqref{eq:fddotvmvs} in the last step.

  Thus, for some constant $C_2$ that depends only on $C_L$, and for $v
  \in B_r(v_*) \cap S_X$, with $r \leq r_1$ and $r_1$ chosen
  sufficiently small, we obtain
  \begin{align}
    \notag
    \| g(v)\|^2 &\geq \| H(v_*) (v-v_*) \|^2 - C_2 \| v - v_*\|^3 \\
    \notag
    &\geq \mu \| H(v_*)^{1/2} (v-v_*) \|^2 - C_2 \|v-v_* \|^3 \\
    \notag
    &\geq \b( \mu - C_2 \mu^{-1}  r \b) \|v-v_*\|_*^2 \\
    \label{eq:contrsd_sph:gbound}
    & \geq \smfrac{\mu}{2} \|v-v_*\|_*^2
  \end{align}
  using  \eqref{eq:mual} and \eqref{contrsd_sph:norm_equiv}.

  {\it Step 4. Short steps. } For $\alpha$ sufficiently small, the
  Armijo condition is in fact not needed, and we can proceed without
  it. From the definition of
  $v_\alpha$ and Taylor's theorem we obtain, for $\alpha \leq \bar\alpha$
  \begin{displaymath}
    v_\alpha - v = \alpha g(v) + O( \alpha^2 \|g(v)\|^2)
  \end{displaymath}
  and hence using  \eqref{eq:contrsd_sph:g_exp}
  \begin{displaymath}
    v_\alpha - v_* = \b[I-\alpha H(v_*)\b] (v-v_*) + O(\alpha \|v-v_*\|^2)
  \end{displaymath}
  Taking the inner product with $H(v_*) (v_\alpha-v_*)$, there exists
  a constant $c_4$ that depends only on the derivatives $F$ in
  $B_1(v_*)$ such that
  \begin{align*}
    \|v_\alpha-v_*\|_*^2 \leq (v-v_*) \cdot \B( H(v_*) \b[I-\alpha
    H(v_*)\b] \B) (v-v_*) + c_4 r \alpha \| v - v_* \|_*^2.
  \end{align*}
  The eigenvalues of $H(v_*) \b[I-\alpha H(v_*)\b] \psi = \tau H(v_*)
  \psi$ are precisely $\tau = 1-\alpha \lambda$ for $\lambda \in
  \sigma(H(v_*))$. Let $\hat{\alpha} > 0$ such that $\tau \in [0, 1)$
  for all $\alpha \leq \hat\alpha$. Then, the largest eigenvalue is
  given by $1 - \alpha \mu$ and we obtain that, for $\alpha \leq
  \hat{\alpha}$,
  \begin{displaymath}
    \|v_\alpha-v_*\|_*^2 \leq (1-\alpha\mu + c_4 \alpha r) \|v-v_*\|_*^2.
  \end{displaymath}
  Choosing $r \leq r_2 \leq r_1$ sufficiently small, with the new
  restrictions depending only on $\mu$ and $c_4$, 
  and using the bound $\sqrt{1-\theta} \leq 1 - \smfrac12 \theta$
  for $\theta \in [0,1]$, we obtain that
  \begin{displaymath}
    \|v_\alpha-v_*\|_* \leq (1-\alpha\smfrac{\mu}{4}) \|v-v_*\|_*.
  \end{displaymath}
  This completes the proof of the Lemma, for the case $\alpha \leq
  \hat{\alpha}$.

  {\it Step 4. Long steps. } Let $\alpha \in [\hat{\alpha},
  \bar\alpha]$, $r \leq r_1$, $v \in B_r(v_*) \cap S_X$, $c_3 \equiv
  c_3(\bar\alpha)$, and $v_\alpha$ satisfying the Armijo condition, then
  \eqref{contrsd_sph:norm_equiv}, \eqref{eq:contrsd_sph:energy_error},
  \eqref{eq:contrsd_sph:bound_step} and \eqref{eq:contrsd_sph:gbound}
  imply
  \begin{align*}
    \b(\smfrac12 - C_1 \mu^{-1} c_3 r\b) \| v_\alpha - v_*\|_*^2 &\leq
    F(v_\alpha) - F(v_*) \\
    &\leq F(v) - F(v_*) - \Theta \alpha \| g(v) \|^2 \\
    &\leq \b(\smfrac12 + C_1 \mu^{-1} r - \Theta \hat{\alpha}
    \smfrac{\mu}{2} \b) \| v - v_*\|_*^2 \,,
  \end{align*}
  that is,
  \begin{displaymath}
    \| v_\alpha - v_*\|_* \leq \bg(\frac{1 + 2C_1 \mu^{-1} r - \Theta \hat{\alpha}
      \mu}{1 - 2 C_1 \mu^{-1} c_3 r}\bg)^{1/2} \| v - v_*\|_*.
  \end{displaymath}
  Thus, choosing $r \leq r_1$, sufficiently small, we obtain again the
  desired contraction.
\end{proof}

\subsection{Proof of Theorem \ref{th:conv_ls}, Case (ii)}
\label{sec:prfls}
%
%
Throughout this proof, we fix an index-1 saddle $(x_*, v_*,
\lambda_*)$, and assume that $h_0$ is small enough so that Proposition
\ref{th:dimer_saddle} ensures the existence of a dimer saddle $(x_h,
v_h, \lambda_h)$ in an $O(h^2)$ neighbourhood of $(x_*, v_*,
\lambda_*)$.

Until we state otherwise (namely, in \S
\ref{sec:prf_case_i}) we assume that $\D_x \E_h(x_k, v_{k-1}) \neq 0$
for all $k$. In particular, the Linesearch Dimer Algorithm is then
well-defined and produces a sequence of iterates $(x_k, v_k)_{k \in
  \N}$. The alternative, Case (i), is treated in
\S\ref{sec:prf_case_i}.


The first step is an error bound on $v_k - v_h$ in terms of $x_k -
x_h$ and the residual of $v_k$.

\begin{lemma}
  \label{th:prfls:L1}
  There exist $r, h_0, C_1 > 0$ such that, for $h \in (0, h_0]$, $x
  \in B_r(x_*)$ and $v \in B_r(v_*)$ with $\|v\| = 1$, we have
  \begin{displaymath}
    \| v - v_h \| \leq \smfrac12 C_1 \b( \|x - x_h \| +  \b\| (I-v\otimes v)
    H_h(x; v) \b\| \b) .
  \end{displaymath}
\end{lemma}
\begin{proof}
  Let $\lambda := H_h(x;v) \cdot v$, then
  \begin{equation}
    \label{eq:prfls:lem1:1}
    \begin{split}
      H_h(x_h; v) &= \lambda v + s, \\
      \smfrac12 \|v\|^2 &= \smfrac12,
    \end{split}
  \end{equation}
  where 
  \begin{displaymath}
    s = \b( H_h(x_h; v) - H_h(x; v) \b) + (I-v\otimes v) H_h(x; v).
  \end{displaymath}
  Since $v_h$ solves \eqref{eq:prfls:lem1:1} with $s = 0$, and since
  \begin{displaymath}
    \|s\| \leq C_2 \b(\| x - x_h \| + \| (I-v\otimes v) H_h(x; v) \| \b),
  \end{displaymath}
  the stated result follows from the Lipschitz continuity of $H_h(\cdot; v)$
  and an application of the inverse
  function theorem, in a similar spirit as the proof in
  \S\ref{sec:proof_dimer_saddle}.
\end{proof}

Next, we present a result ensuring that the rotation step of Algorithm
3 not only terminates but also produces a new dimer orientation $v_k$
which remains in a small neighbourhood of the ``exact'' orientation
$v_h$.

\begin{lemma}
  \label{th:prfls:rotation_lemma}
  There exist $r, h_0, C_2 > 0, C_3 \geq 1$ such that, if $h \in (0,
  h_0]$, $x_k \in B_r(x_*)$, $v_{k-1} \in B_{C_3 r}(v_*)$,
  $\|v_{k-1}\| = 1$, then Step (3) of Algorithm 3 terminates with
  outputs $v_k \in B_{C_3 r}(v_*)$, $\|v_k \| = 1$, $\beta_k > 0$,
  satisfing
  \begin{align}
    \label{eq:prfls:rot_bound_1}
    \| v_k - v_h \| &\leq C_2 \b(\| x_k - x_h \| + h^2 \| v_{k-1}-v_h\|
    \b).
  \end{align}
\end{lemma}
\vspace{-4mm}
\begin{proof}
  Let $G(v) := h^{-2}(\E_h(x_k; v) - \E_h(x_k; V(x_k)))$, then each
  step of the Rotation Algorithm is a steepest descent step of $G$ on
  the manifold $S_X := \{\|v\|=1\}$.  We need to ensure that these
  iterations do not ``escape'' from the minimiser.

  Lemma \ref{th:contrSX} (with $F(v) = G(v)$ and $v_* \equiv V(x_k)$) 
  implies that each
  such step is a contraction towards $V(x_k)$ with respect to the norm
  $\|\cdot\|_H$ induced by the operator
  \begin{displaymath}
    H := (I-V\otimes V) \D^2 G(V) (I-V\otimes V) - (\D G(V) \cdot V) I,
  \end{displaymath}
  where $V \equiv V(x_k)$; provided that $r$ is sufficiently small and
  $H$ is positive definite.

  To see that the latter is indeed true, we recall from 
 \eqref{eq:defn:errorH} and \eqref{eq:defn:dtwoh} that
  \begin{align*}
  \D G(V(x_k))  &= \D^2 E(x_k) V(x_k) + O(h^2) \quad \mbox{and} \quad
  \D^2 G(V(x_k)) = \D^2 E(x_k) + O(h^2)
  \end{align*}
  and from Proposition \ref{th:dimer_saddle} and Lemma
  \ref{th:aux_lemma_Eh} that 
  \begin{equation}
  V(x_k) = v_* + O(h^2+r),
  \label{eq:vdiff}
  \end{equation}
  and hence,
  \begin{align*}
    H &= (I-v_*\otimes v_*) \D^2 E(x_*) (I-v_*\otimes v_*) - \b( (\D^2
    E(x_*) v_*) \cdot v_* \b) I  + O(h^2+r)\\
    &=  (I-v_*\otimes v_*) \D^2 E(x_*) - \lambda_*
    I + O(h^2 + r).
  \end{align*}
  Since $(x_*, v_*, \lambda_*)$ is an index-1 saddle, $(I-
  v_*\otimes v_*) \D^2 E(x_*)$ is positive definite in
  $\{v_*\}^\perp$, and $\lambda_* < 0$. Thus, for $h, r$ sufficiently
  small, $H$ is positive definite as required.

  From Lemma \ref{th:contrSX}, it follows that all iterates
  $v^{(j)}_k$ of the Rotation Algorithm satisfy $\| v^{(j)}_k - V(x_k)
  \|_H \leq \|v_{k-1} - V(x_k) \|_H$. Since the eigenvalues of $H$ are
  uniformly bounded below and above, the norms $\|\cdot\|_H,
  \|\cdot\|$ are equivalent, and hence in particular
   \begin{displaymath}
    \| v_k - V(x_k) \| \leq C_7 \|v_{k-1} - V(x_k) \|
   \leq C_7 ( \|v_{k-1} - v_*  \| + \| V(x_k) - v_*\| )
   = O( h^2 + r )
   \end{displaymath}
   for some constant $C_7 > 0$, since
   $v_{k-1} \in B_{C_3 r}(v_*)$ and using \eqref{eq:vdiff}.
   Combining this with \eqref{eq:vdiff} and choosing $h_0^2 \leq r$, we deduce
   that the Rotation Algorithm terminates with an iterate $v_k$ such that
   \begin{displaymath}
    \|v_* - v_k \| \leq \|v_* - V(x_k)\| + \|v_k - V(x_k)\| 
     \leq C_4 r
   \end{displaymath}
   for some constant that depends only on $r$ but is independent of
   $v_{k-1}$ and remains bounded as $r \to 0$.  

  At termination the Rotation Algorithm guarantees the estimate
  \begin{displaymath}
    \b\| (I-v_k\otimes v_k) H_h(x_k; v_k) \b\| \leq \| \D_x \E_h(x_k,
    v_{k-1}) \|.
  \end{displaymath}
  We set $x^t = (1-t) x_h + t x_k$, $v^t = v_h + t v_{k-1}$ and
  expand
  \begin{align*}
    \b\|\D_x \E_h(x_k,
    v_{k-1})\b\| &= \bg\|\int_0^1 \B( \D_x^2\E_h(x^t, v^t) (x_k-x_h) +
    \D_v\D_x\E_h(x^t,v^t) (v_{k-1}-v_h) \B) \dt\bg\| \\
    &\leq C_2' \b( \| x_k - x_h \| + h^2 \| v_{k-1}-v_h \| \b).
  \end{align*}
  Combined with Lemma \ref{th:prfls:L1} this yields the estimate
  \eqref{eq:prfls:rot_bound_1}. 

  The statement that $v_k \in B_{C_3 r}(v_*)$ (instead of only $B_{C_4
    r}(v_*)$) is an immediate consequence of
  \eqref{eq:prfls:rot_bound_1} by ensuring that $C_3 \geq C_2 + C_3
  h^2 + C' h^4$, where $\| v_h - v_* \| \leq C' h^2$ for all $h \leq
  h_0$ from Proposition~\ref{th:dimer_saddle}.  While there is an
  interdependence between $C_3$ and $C_2$, for $r$ and $h_0$
  sufficiently small, this is clearly achievable.
\end{proof}

We now establish the existence of a minimiser of the auxiliary
functional $F_k$ under the conditions ensured by the rotation step of
Algorithm 3.

\begin{lemma}
  \label{th:prfls:L2}
  Under the conditions of Lemma \ref{th:prfls:rotation_lemma},
  possibly after choosing a smaller $r, h_0$, there exists a constant
  $C_4 > 0$, such that the functional $F_k$ defined in
  \eqref{eq:defn_aux_fcnl} has a unique minimiser $y_k \in B_r(x_*)$
  satisfying
  \begin{equation}
    \| y_k - x_h \| \leq C_4 ( r_k^2 + h^2 r_k + h^4 s_{k-1}).
    \label{eq:ykmxh}
  \end{equation}
\end{lemma}
\vspace{-5mm}
\begin{proof}
  We begin by estimating the residual
  \begin{align*}
    \D F_k(x_h) &= \D_x \E_h(x_h, v_k) - 2 (\D_x \E_h(x_k, v_k) \cdot
    v_k) v_k + 2 \lambda_k ((x_k-x_h) \cdot v_k) v_k,
  \end{align*}
  where $\lambda_k = H_h(x_k; v_k) \cdot v_k$. We consider each constituent term
  in this expression in turn; we expand about $(x_h, v_h)$, and use the 
  identities \eqref{eq:defn:Hh},  \eqref{eq:dimer_saddle_system} and
 \eqref{eq:nablavxh}
  This gives
  \begin{align*}
   v_k &= v_h + O(s_k) \\
   \D_x \E_h(x_h, v_k) &=
     \D_x \D_v \E_h(x_h, v_h) (v_k - v_h) + O( s_k^2 ) \\
  \D_x \E_h(x_k, v_k) &=
    \D_x^2 \E_h(x_h, v_h) (x_k - x_h) +
    \D_x\D_v \E_h(x_h, v_h)(v_k - v_h) + O(r_k^2) + O(s_k^2) \\
   &= \D_x^2 \E_h(x_h, v_h) (x_k - x_h) + O(h^2 s_k)
    + O(r_k^2) + O(s_k^2),  \\
\D_x \E_h(x_k, v_k) \cdot v_k v_k &=
  ( \D_x^2 \E_h(x_h, v_h) (x_k - x_h) +
    \D_x\D_v \E_h(x_h, v_h)(v_k - v_h) ) \cdot v_h v_h \\
   & \qquad + O(r_k^2) + O(s_k^2) + O( r_k s_k) \\
    &=  \D_x^2 \E_h(x_h, v_h) (x_k - x_h) \cdot v_h v_h
   + O(h^2 s_k) + O(r_k^2) + O(s_k^2) + O( r_k s_k) 
   \\
   H_h(x_k;v_k) &= H_h(x_h;v_h) + O(r_k) + O(s_k) \\
   \lambda_k &= \lambda_h + 
    v_k \cdot H_h(x_k;v_k) - v_h \cdot H_h(x_h;v_h) 
   = \lambda_h + O(r_k) + O(s_k) \\
   \lambda_k ((x_k-x_h) \cdot v_k) v_k 
   &= (\lambda_h + O(r_k) + O(s_k))((x_k-x_h)\cdot v_k) v_k \\
   & = \lambda_h ((x_k-x_h)\cdot v_h) v_h + O(r_k^2) + O(r_k s_k).
  \end{align*}
Thus since \eqref{eq:dimer_saddle_error} and our assumption that
$v_{k-1} \in B_{C_3 r}(v_*)$ ensure that $s_{k-1} = O(1 + h_0^2)$, while
\eqref{eq:prfls:rot_bound_1} implies that
 $s_k = O(r_k) + O(h^2 s_{k-1})$, we combine the above to obtain
  \begin{align*}
    \D F_k(x_h) &= 
    - 2 \b[ (\D_x^2 \E_h(x_h, v_h) (x_k-x_h))
    \cdot v_h \b] v_h + 2 \lambda_h ((x_k-x_h) \cdot v_h) v_h \\
    & \qquad + O\b(r_k^2 + h^2 r_k + h^4 s_{k-1}\b),
  \end{align*}
  Next, we note that, by definition of $\E_h$, 
 $\D_x^2 \E_h(x_h, v_h) v_h = \D^2 E(x_h)v_h + O(h^2)$,
  and thus from \eqref{eq:defn:errorH} that
  $\D_x^2 \E_h(x_h, v_h) v_h = H_h(x_h; v_h) +  O(h^2)$. 
 Hence applying \eqref{eq:dimer_saddle_system}, 
  \begin{align}
    \D F_k(x_h) &= \b[- 2 H_h(x_h; v_h) \cdot (x_k - x_h)  + 2 \lambda_h
    (x_k - x_h) \cdot v_h \b] v_h + O\b(r_k^2 + h^2 r_k + h^4 s_{k-1}\b) 
    \notag \\
    &= O\b(r_k^2 + h^2 r_k + h^4 s_{k-1}\b). \label{eq:dFk}
  \end{align}
  Finally, we observe that $\D^2 F_k(x_h)$ is positive definite, since
  \begin{align}
    \notag \D^2 F_k(x_h) &= \D_x^2 \E_h(x_h, v_k) - 2 \lambda_k v_k
    \otimes
    v_k \\
    \notag
    &= \D_x^2 \E_h(x_h, v_h) - 2 \lambda_h v_h \otimes v_h + O(r_k) \\
    \label{eq:expansion_D2Fk}
    &= \D^2 E(x_*) - 2 \lambda_* v_* \otimes v_* + O(h^2+r_k),
  \end{align}
  which immediately implies that, for $r, h_0$ sufficiently small,
  $\D^2 F_k(x_h)$ is an isomorphism with uniformly bounded inverse. 

  Thus an application of the inverse function theorem to $\D F_k$ at $y_k$
  using \eqref{eq:dFk}  yields the stated result.
\end{proof}

We now turn towards analysing the linesearch for $x$. Recall the
definition of the energy norm $\|u\|_* := \sqrt{u \cdot
  ((I-2v_*\otimes v_*) \D^2 E(x_*) u)}$, which is equivalent to
$\|\cdot\|$. In particular, 
\begin{equation}
\label{eq:equiv}
\mu^{1/2} \|u\| \leq \|u\|_* \leq \|\D^2 E(x_*)\| \|u\|
\quad \mbox{where} \quad
\mu: = \min( - \lambda_*,\mu_*) > 0.
\end{equation}

\begin{lemma}
  \label{th:prfls:contr_x}
  There exists $r, h_0, \underline\alpha \in (0, \alpha_0]$ and
  $\gamma_* \in (0, 1)$, such that, if $h \in (0, h_0]$, $x_k \in
  B_r(x_*), v_k \in B_{C_3r}(v_*)$ and $\alpha_{k-1} \geq
  \underline\alpha$, then 
  \begin{displaymath}
    \alpha_k \geq \underline\alpha \qquad \text{and} \qquad
    \| x_{k+1} - y_k \|_* \leq \gamma_* \|x_k - y_k \|_*,
  \end{displaymath}
  where $y_k$ is the minimiser of $F_k$ established in
  Lemma \ref{th:prfls:L2}.
\end{lemma}
\begin{proof}
  We begin by noting that, for any $r > 0$, the norms $\|\D^2
  F_k(x)\|$ are uniformly bounded among all choices of $x_k \in
  B_{r}(x_*)$, $x \in B_{r+1}(x_*)$. This is straightforward to
  establish.

  Therefore, there exists $\underline{\alpha} > 0$ such that, for $x_k
  \in B_{r}(x_*)$ and for any $\alpha \in (0, 2\underline\alpha]$, the
  conditions in Step (6) of Algorithm 3 are met (this includes an
  Armijo condition for $F_k$) since $\nabla F_k$ is Lipschitz in a
  neighbourhood of $x_k$ \cite[Thm.2.1]{GoulLeyf03}.  It is no
  restriction of generality to require $\underline\alpha \leq
  \alpha_0$. In particular, $\alpha_k \geq \underline\alpha$.

  For $r, h_0$ sufficiently small, we have $y_k \in B_r(x_*)$ as
  well. Upon choosing $r$ sufficiently small, $u\cdot( \D^2 F_k(y) u)
  \geq \mu/2 \|u\|^2$ for all $u \in X$, $y \in B_{r}(x_*)$. Thus, we
  can apply Lemma~\ref{th:contrsd_eucl} (with $x_* \equiv y_k$) to
  deduce that, for $r$ sufficiently small, the step $x_{k+1} = x_k -
  \alpha_k \D F(x_k)$ is a contraction with a constant $\gamma_1$ that
  is independent of $x_k, v_k$. That is,
  \begin{displaymath}
    (x_{k+1} - y_k) \cdot \b[\D^2 F_k(y_k) (x_{k+1}-y_k) \b] \leq
    \gamma_1^2 (x_{k} - y_k) \cdot \b[\D^2 F_k(y_k) (x_{k}-y_k) \b], 
  \end{displaymath}
  Recalling from \eqref{eq:ykmxh} and \eqref{eq:expansion_D2Fk} 
  that $\D^2 F_k(y_k) = (I-2
  v_*\otimes v_*) \D^2 E(x_*) + O(r+h^2)$ we find that, for $r, h_0$
  sufficiently small, 
  \begin{equation}
    \label{eq:prfls:contr1}
    \| x_{k+1} - y_k \|_* \leq \gamma_* \| x_k - y_k \|_*,
  \end{equation}  
  where $\gamma_* \in [\gamma_1, 1)$, again independent of $x_k, v_k$,
  but depending on $r, h_0$.
\end{proof}
\noindent We have now assembled all prerequisites required to complete the proof
of Theorem~\ref{th:conv_ls}. 

Inspired by Lemma \ref{th:prfls:contr_x},
our aim is to prove that, for $r$ sufficiently small, there exists
$\gamma \in (0, 1)$ such that, for all $j \geq 0$,
\begin{equation}
  \label{eq:prfls:main_goal}
  r_j^* + h^2 s_{j-1} \leq \gamma^j (r_0^*+h^2 s_{-1}) =: \gamma^j t_0,
\end{equation}
where $\gamma := \smfrac12 (\gamma_*+1)$, $r_k^* :=
\|x_k - x_h\|_*$ and $s_k := \|v_k - v_h \|$.

A consequence of \eqref{eq:prfls:main_goal} would be that there exists
a constant $c$ such that $\|x_j-x_*\| \leq c r =: \hat{r}$. Thus,
under the assumptions of the Theorem, let $r, h_0$ be chosen
sufficiently small so that Proposition \ref{th:dimer_saddle}, and
Lemmas \ref{th:prfls:L1}, \ref{th:prfls:rotation_lemma},
\ref{th:prfls:L2} and~\ref{th:prfls:contr_x} apply with $r$ replaced
by $\hat{r}$. 

We now begin the induction argument adding to
\eqref{eq:prfls:main_goal} the conditions that 
\begin{equation}
  \label{eq:prflj:main_goal_add}
  v_{j-1} \in B_{C_3
    r}(v_*) \qquad \text{and} \qquad \alpha_{j} \geq \underline\alpha,
\end{equation}
where $C_3 \geq 1$ is the constant from Lemma
\ref{th:prfls:rotation_lemma} and $\underline\alpha$ the constant from
Lemma \ref{th:prfls:contr_x}. Clearly \eqref{eq:prfls:main_goal} and
\eqref{eq:prflj:main_goal_add} hold for $j = 0$. Suppose that they
hold for $j = 0, \dots, k$, where $k \geq 0$.

The choice of $r$ implies that $x_k \in B_r(x_*)$ again, and Lemma
\ref{th:prfls:rotation_lemma} implies that $v_k \in B_{C_3
  r}(v_*)$. Thus, the first condition in
\eqref{eq:prflj:main_goal_add} is established for $j = k+1$.

Applying Lemma \ref{th:prfls:contr_x} we obtain the second condition
in \eqref{eq:prflj:main_goal_add} for $j = k+1$, and in addition that
\begin{displaymath}
  \| x_{k+1} - y_k \|_* \leq \gamma_* \| x_k - y_k \|_*,
\end{displaymath}
where $y_k$ is the minimiser of $F_k$ established in Lemma
\ref{th:prfls:L2}.  Using \eqref{eq:prfls:contr1}, the fact that
$\gamma_* < 1$ and Lemma \ref{th:prfls:L2} 
we therefore deduce that there exists a constant
$C_5$ which depends on $C_4$ and on the norm-equivalence between
$\|\cdot\|$ and $\|\cdot\|_*$, such that
\begin{align*}
  \| x_{k+1} - x_h \|_* &\leq \|x_{k+1} - y_k \|_* + \| y_k - x_h
  \|_* \\
  & \leq \gamma_* \|x_k - y_k \|_* + \| y_k - x_h \|_* \\
  & \leq \gamma_* \|x_k - x_h \|_* + 2 \| y_k - x_h \|_* \\ 
  & \leq (\gamma_* + C_5 h^2 + C_5 r_k) \|x_k - x_h \|_* + C_5 h^4 \|
  v_{k-1} - v_h \|.
\end{align*}

Adding $h^2 \| v_k - v_h \|$ to both sides of the inequality and
applying \eqref{eq:prfls:rot_bound_1} and \eqref{eq:equiv} we thus obtain
\begin{align*}
  r_{k+1}^* + h^2 s_k &\leq (\gamma_* + C_5 h^2 + C_5 r_k) r_k^* +
  h^2 s_k + C_5 h^4 s_{k-1} \\
  &\leq \b(\gamma_* + C_5 h^2 +\mu^{-1/2}C_2 h^2 + C_5 (c+1) r\b) r_k^* + 
  (C_5+C_2) h^4  s_{k-1}.
\end{align*}
Recalling that $\gamma = \smfrac12 (\gamma_* + 1)$, choosing $h_0, r$
sufficiently small, we obtain that 
\begin{displaymath}
  r_{k+1}^* + h^2 s_k \leq \gamma (r_k^* + h^2 s_{k-1}).
\end{displaymath}
This establishes \eqref{eq:prfls:main_goal} for $j = k+1$ and thus
completes the induction argument.

In summary, we have proven that \eqref{eq:prfls:main_goal} and
\eqref{eq:prflj:main_goal_add} hold for all $j \geq 0$.  As a first
consequence, we obtain that $r_k := \|x_k - x_h\| \leq 
\mu^{-1/2} \|\D^2 E(x_*)\| \gamma^k (r_0 + h^2 s_{-1})$
using \eqref{eq:equiv},
which in particular establishes the first part of \eqref{eq:conv_ls_maineq}.

To obtain a convergence rate for $v_k$ we combine
\eqref{eq:prfls:rot_bound_1} and \eqref{eq:prfls:main_goal}, to obtain
\begin{displaymath}
  \| v_k - v_h \| \leq C_6 (r_k^* + h^2 s_{k-1}) \leq C_6 \gamma^k t_0
\leq C_6 \|\D^2 E(x_*)\|  \gamma^k (r_0 + h^2 s_{-1}),
\end{displaymath}
for a constant $C_6$. Choosing $C = 2 \max( C_6, \mu^{-1/2}) \|\D^2 E(x_*)\|$
completes the proof of Theorem
\ref{th:conv_ls}.





\subsubsection{Proof of Case (i)}
\label{sec:prf_case_i}
The proof of Case (ii) establishes that, for as long as we have $\D_x
  \E_h(x_k, v_{k-1}) \neq 0$, the iterates are well-defined and $\|
  x_k - x_h \| + \| v_k - v_h \| \leq C r$ for some suitable constant
  $C$. We now drop this assumption and instead suppose that, at the
  $\ell$th iterate, $\D_x \E_h(x_\ell, v_{\ell-1}) = 0$. In this case,
  we can apply the following lemma.


\begin{lemma}
  \label{eq:min_Eh_x}
  Let $(x_*, v_*, \lambda_*)$ be an index-1 saddle, then there exist
  $r, h_0, C > 0$ such that, for all $h \in (0, h_0]$ and for all $v
  \in S_X$, there exists a unique $x_{h,v} \in B_r(x_*)$ such that
  $\D_x \E_h(x_{h,v}, v) = 0$. Moreover, $\| x_{h,v} - x_h \| \leq C h^2$.
\end{lemma}
\begin{proof}
  This is an immediate corollary of \eqref{eq:err_DxEh} and the
  inverse function theorem.
\end{proof}

Since $\| x_\ell - x_h \| \leq C r$, Lemma \ref{eq:min_Eh_x} implies
that, in fact $\|x_\ell-x_h \| \leq C' h^2$ for some other constants
$C'$, provided that $r, h$ are chosen sufficiently small.

This concludes the proof of Theorem \ref{th:conv_ls}, Case (i).

\bibliographystyle{plain}
\bibliography{qc}

\begin{thebibliography}{10}

\bibitem{AbsiMahoSepu08}
P.-A. Absil, R.~Mahony, and R.~Sepulchre.
\newblock {\em Optimization Algorithms on Matrix Manifolds}.
\newblock Princeton University Press, Princeton, USA, 2008.

\bibitem{AbsiMahoTrum13}
P.-A. Absil, R.~Mahony, and J.~Trumpf.
\newblock An extrinsic look at the {R}iemannian {H}essian.
\newblock In F.~Nielsen and F.~Barbaresco, editors, {\em Geometric Science of
  Information}, number 8005 in Lecture Notes in Computer Science, pages
  361--368, Heidelberg, Berlin, New York, 2013. Springer Verlag.

\bibitem{Banerjee:jpc1985}
A.~Banerjee, N.~Adams, J.~Simons, and R.~Shepard.
\newblock Search for stationary points on surfaces.
\newblock {\em The Journal of Physical Chemistry}, 89:52--57, 1985.

\bibitem{Barkema:prl1996}
G.T. Barkema and N.~Mousseau.
\newblock Event-based relaxation of continuous disordered systems.
\newblock {\em Physical Review Letters}, 77:4358, 1996.

\bibitem{Barkema:cms2001}
G.T. Barkema and N.~Mousseau.
\newblock The activation-relation technique: an efficient algorithm for
  sampling energy landscapes.
\newblock {\em Computational Materials Science}, 20(3--4):285--292, 2001.

\bibitem{Cances:jcp2009}
E.~Cances, F.~Legoll, M.C. Marinica, K.~Minoukadeh, and F.~Willaime.
\newblock Some improvements of the activation-relaxation technique method for
  finding transition pathways on potential energy surfaces.
\newblock {\em The Journal of Chemical Physics}, 130(114711), 2009.

\bibitem{Cerjan:jcp1981}
C.J. Cerjan and W.H. Miller.
\newblock On finding transition states.
\newblock {\em Journal of Chemical Physics}, 75:2800, 1981.

\bibitem{GoulLeyf03}
N.~I.~M. Gould and S.~Leyffer.
\newblock An introduction to algorithms for nonlinear optimization.
\newblock In A.~W.~Craig, J. F.~Blowey and T.~Shardlow, editors, {\em Frontiers
  in Numerical Analysis (Durham 2002)}, pages 109--197, Heidelberg, Berlin, New
  York, 2003. Springer Verlag.

\bibitem{HenkJons:jcp1999}
G.~Henkelman and H.~J\'{o}nsson.
\newblock A dimer method for finding saddle points on high dimensional
  potential surfaces using only first derivatives.
\newblock {\em Journal of Chemical Physics}, 111(5):7010--7022, 1999.

\bibitem{Heyden:jcp2005}
A.~Heyden, A.~T. Bell, and F.~J. Keil.
\newblock Efficient methods for finding transition states in chemical
  reactions: Comparison of improved dimer method and partitioned rational
  function optimization method.
\newblock {\em Journal of Chemical Physics}, 123(224101), 2005.

\bibitem{Jonsson:1998}
H.~J\'{o}nsson, G.~Mills, and K.~W. Jacobsen.
\newblock Nudged elastic band for finding minimum energy paths of transitions.
\newblock In G.~Ciccotti B.~J.~Berne and D.~F. Coker, editors, {\em Classical
  and quantum dynamics in condensed phase simulations}, volume 385. World
  Scientific, 1998.

\bibitem{Kastner:jcp2008}
J.~K{\"a}stner and P.~Sherwood.
\newblock Superlinearly converging dimer method for transition state search.
\newblock {\em The Journal of Chemical Physics}, 128(014106), 2008.

\bibitem{Lanc50}
C.~Lanczos.
\newblock An iteration method for the solution of the eigenvalue problem of
  linear differential and integral operators.
\newblock {\em Journal of research of the National Bureau of Standards B},
  45:225--280, 1950.

\bibitem{LiuNoce89}
D.~Liu and J.~Nocedal.
\newblock On the limited memory {BFGS} method for large scale optimization.
\newblock {\em Mathematical Programming, Series~B}, 45(3):503--528, 1989.

\bibitem{Machado-Charry:jcp2011}
E.~Machado-Charry, L.K. Beland, D.~Caliste, Luigi Genovese, T.~Deutsch,
  N.~Mousseau, and P.~Pochet.
\newblock Optimized energy landscape exploration using the ab initio based
  activation-relaxation technique.
\newblock {\em Journal of Chemical Physics}, 135(034102), 2011.

\bibitem{Marinica:prb2011}
M.C. Marinica, F.~Willaime, and N.~Mousseau.
\newblock Energy landscape of small clusters of self-interstitial dumbbells in
  iron.
\newblock {\em Physical Review B}, 83(094119), 2011.

\bibitem{Mousseau:jamop2012}
N.~Mousseau, L.K. Beland, P.~Brommer, J.F. Joly, F.~El-Mellouhi,
  E.~Machado-Charry, M.C. Marinica, and P.~Pochet.
\newblock The activation-relaxation technique: Art nouveau and kinetic art.
\newblock {\em Journal of Atomic, Molecular and Optical Phsyics}, 2012(952278),
  2012.

\bibitem{Murtagh:cj1970}
B.~A. Murtagh and R.~W.~H. Sargent.
\newblock Computational experience with quadratically convergent minimisation
  methods.
\newblock {\em The Computer Journal}, 13:185--194, 1970.

\bibitem{NocedalWright}
J.~Nocedal and S.~J. Wright.
\newblock {\em Numerical Optimization}.
\newblock Springer, 1999.

\bibitem{OlsenKroes:jcp2004}
R.~A. Olsen, G.~J. Kroes, G.~Henkelman, A.~Arnaldsson, and H.~Jónsson.
\newblock Comparison of methods for finding saddle points without knowledge of
  the final states.
\newblock {\em Journal of Chemical Physics}, 121:9776, 2004.

\bibitem{Simons:jpc1983}
J.~Simons, P.~Joergensen, H.~Taylor, and J.~Ozment.
\newblock Walking on potential energy surfaces.
\newblock {\em The Journal of Physical Chemistry}, 87:2745--2753, 1983.

\bibitem{Uberuaga:2005}
B.~P. Uberuaga, F.~Montalenti, T.~C. Germann, and A.~F. Voter.
\newblock Accelerated molecular dynamics methods.
\newblock In S.~Yip, editor, {\em Handbook of Materials Modelling, Part A-
  Methods}, page 629. Springer, 2005.

\bibitem{MikhRedkoPere87}
A.~E.~Perekatov V.~S.~Mikhalevich, N. N.~Redkovskii.
\newblock Methods of minimization of functions on a sphere and their
  applications.
\newblock {\em Cybernetics and Systems Analysis}, 23(6):721--730, 1987.

\bibitem{Voter:prl1997}
A.~F. Voter.
\newblock Accelerated molecular dynamics of infrequent events.
\newblock {\em Physical Review Letters}, 78(3908), 1997.

\bibitem{Voter:2007}
A.~F. Voter.
\newblock Introduction to the kinetic {M}onte {C}arlo method.
\newblock In K.~E. Sickafus, E.~A. Kotomin, and B.~P. Uberuaga, editors, {\em
  Radiation Effects in Solids}, volume 235 of {\em NATO Science Series}, pages
  1--23. Springer Netherlands, 2007.

\bibitem{Weinan:prb2002}
E.~Weinan, W.~Ren, and E.~Vanden-Eijnden.
\newblock String method for the study of rare events.
\newblock {\em Physical Review B}, 66(052301), 2002.

\bibitem{Weinan:jcp2007}
E.~Weinan, W.~Ren, and E.~Vanden-Eijnden.
\newblock Simplified and improved string method for computing the minimum
  energy path in barrier-crossing events.
\newblock {\em Journal of Chemical Physics}, 126(164103), 2007.

\bibitem{ZhangDu:sinum2012}
J.~Zhang and Q.~Du.
\newblock Shrinking dimer dynamics and its applications to saddle point search.
\newblock {\em SIAM Journal of Numerical Analysis}, 50(4):1899--1921, 2012.

\end{thebibliography}
\end{document}